\documentclass[11pt]{amsart}
\usepackage{amsthm}

\usepackage[latin1]{inputenc}
\usepackage{subcaption}
\usepackage{color}
\usepackage{amsmath}
\usepackage{amsthm}
\usepackage{amstext}
\usepackage{amssymb}
\usepackage{amsfonts}
\usepackage{graphicx}
\usepackage{young}
\usepackage{multicol}
\usepackage{mathrsfs}
\usepackage{stmaryrd}
\usepackage{bbm}
\usepackage[all]{xy}
\usepackage{mathabx}
\usepackage{mathtools}
\usepackage{tabularx}
\usepackage{array}
\usepackage{commath}
\usepackage{ytableau}
\usepackage{caption}
\usepackage{tikz, tikz-3dplot, pgfplots}
\usepackage{ulem}
\usetikzlibrary[positioning,patterns]
\pgfplotsset{compat=1.8}
\tikzstyle{bsq}=[rectangle, draw, thick, minimum width=1cm, minimum height=1cm] 
\tikzstyle{bver}=[rectangle, draw, thick, minimum width=1cm, minimum height=2cm]
\tikzstyle{bhor}=[rectangle, draw, thick, minimum width=2cm, minimum height=1cm]

\theoremstyle{plain}
\theoremstyle{definition}
\newtheorem{theorem}{Theorem}[section]

\newtheorem{definition}[theorem]{Definition}

\newtheorem{example}[theorem]{Example}
\newtheorem{proposition}[theorem]{Proposition}
\newtheorem{corollary}[theorem]{Corollary}

\DeclareMathAlphabet{\mathpzc}{OT1}{pzc}{m}{it}

\usepackage{amssymb,amsmath,tabularx,graphicx}
\usepackage{tikz}
\usetikzlibrary[calc,intersections,through,backgrounds,arrows,decorations.pathmorphing]

\usepackage[colorinlistoftodos]{todonotes}

\begin{document}
\date{}
\title{Set-valued domino tableaux and shifted set-valued domino tableaux}
\author{Florence Maas-Gari\'epy}
\address{Laboratoire de Combinatoire et d'Informatique Math\'ematique,
Universit\'e du Qu\'ebec \`a Montr\'eal}
\email{maas-gariepy.florence@courrier.uqam.ca}
\author{Rebecca Patrias}
\address{Laboratoire de Combinatoire et d'Informatique Math\'ematique,
Universit\'e du Qu\'ebec \`a Montr\'eal}
\email{patriasr@lacim.ca}

\begin{abstract}
We prove $K$-theoretic and shifted $K$-theoretic analogues of the bijection of Stanton and White between domino tableaux and pairs of semistandard tableaux. As a result, we obtain product formulas for  pairs of stable Grothendieck polynomials and pairs of $K$-theoretic $Q$-Schur functions.
\end{abstract}

\maketitle

\section{Introduction}
Recall that a \textit{partition} is a finite, nonincreasing sequence of positive integers $\lambda=(\lambda_1,\ldots,\lambda_t)$ 
and that any partition can be identified with the corresponding \textit{Young diagram}---a left-justified array of boxes with $\lambda_i$ boxes in the $i$th row from the top. A filling of the boxes of a Young diagram with non-negative integers such that entries weakly increase across rows and strictly increase down columns   gives a \textit{semistandard Young tableau}. The \textit{Schur functions} are symmetric functions that are indexed by partitions. Each element of the set of Schur functions can be defined as a weighted generating function of semistandard Young tableaux of the corresponding partition shape. The set of Schur functions forms a linear basis for the ring of symmetric functions and appears naturally in many areas of mathematics including representation theory, Schubert calculus, and gauge theory. 

Of particular interest is the question of how to express a product of two Schur functions since the answer has meaning in the fields mentioned above. One way to answer this question for certain products is to consider \textit{domino tableaux}, where a domino tableau is an array of dominoes ($2\times 1$ and $1\times 2$ pieces) in the shape of a Young diagram, where each domino is filled with a positive integer, rows are weakly increasing and columns are strictly increasing. In \cite{stanton1985schensted}, D. Stanton and D. White prove that for arbitrary partitions $\mu$ and $\nu$, the product $s_\mu s_\nu$ can be written as a sum of weighted generating functions of domino tableaux. Their proof was later simplified by C. Carr\'e and B. Leclerc \cite{carre1995splitting}.

A well-known analogue of the Schur functions is the set of \textit{$Q$-Schur functions} $\{Q_\lambda\}$, which are indexed by partitions $\lambda=(\lambda_1,\lambda_2,\ldots,\lambda_t)$ with $\lambda_t\geq t$. 
In \cite{chemli}, Z. Chemli introduces the notion of a \textit{shifted domino tableau} and proves the analogue of the Stanton--White result in this setting. Namely, he proves that for shifted partitions $\mu$ and $\nu$, the product $Q_\mu Q_\nu$ can be written as a sum of weighted generating functions of shifted domino tableaux.

In addition to this shifted analogue of the Schur functions, there is also a natural $K$-theoretic analogue called the \textit{stable Grothendieck polynomials}, denoted $G_\lambda$ \cite{fomin1996yang,buch2002littlewood}. Combinatorially, we obtain the stable Grothendieck polynomials by allowing finite, non-empty subsets of positive integers to fill the boxes of a Young diagram instead of only allowing single entries. The stable Grothendieck polynomials are called $K$-theoretic analogues because where there is a deep connection between Schur functions and cohomology of the Grassmannian, there is the same connection between stable Grothendieck polynomials and $K$-theory of the Grassmannian. A reader unfamiliar with cohomology theory and $K$-theory need not worry; we will only address the combinatorial properties of these symmetric functions. There is also a natural $K$-theoretic analogue of the $Q$-Schur functions, i.e., a natural shifted analogue of the stable Grothendieck polynomials \cite{ikeda2013k, graham2015excited}, denoted $GQ_\lambda$. The $GQ_\lambda$ appear in the study of the $K$-theory of the Lagrangian Grassmannian.

In this paper, we prove the $K$-theoretic analogue of both the Stanton--White result and the Chemli result, thus obtaining product formulas for pairs $G_\mu G_\nu$ and $GQ_\mu GQ_\nu$. Note that $G_\mu G_\nu$ and $GQ_\mu GQ_\nu$ expand positively in terms of the $G_\lambda$'s and $GQ_\lambda$'s, respectively; however, no combinatorial description of this $GQ_\mu GQ_\nu$ expansion is known.
To obtain our results, we introduce the notions of \textit{set-valued domino tableaux} and \textit{shifted set-valued domino tableaux}.

The paper proceeds as follows. In Section~\ref{sec:prelim}, we review the necessary combinatorial background for the rest of the paper: tableaux, Young diagrams, symmetric functions, and Schur functions. Section~\ref{sec:dominotableaux} gives an introduction to domino tableaux, and Section~\ref{sec:bijection} explains the bijection that leads to the result of Stanton and White. In Section~\ref{sec:K-theory}, we introduce the stable Grothendieck polynomials and set-valued domino tableaux, and we prove the $K$-theoretic analogue of the Stanton--White result. Section~\ref{sec:shifted} gives the necessary background on $Q$-Schur functions, shifted Young tableaux and shifted domino tableaux, and reviews the result of Chemli. We conclude in Section~\ref{sec:shiftedK} by proving the shifted $K$-theoretic analogue of Chemli's result.

\section{Preliminaries}\label{sec:prelim}
We begin by reviewing basic notions related to symmetric functions, partitions, and Young tableaux. We refer the reader to \cite{stanley} for a more in-depth study of these topics. 

\subsection{Partitions and tableaux}\label{sec:tableaux}

A \textit{partition} is a finite, nonincreasing sequence of positive integers $\lambda=(\lambda_1,\ldots,\lambda_k)$. We say that $\lambda$ is a \textit{partition of $n$}, written $\lambda\vdash n$ or $|\lambda|=n$, when $\lambda_1+\cdots+\lambda_k=n$. For example, $(2,1,1)$ is a partition of 4, and there are five partitions of 4 in total. To each partition, we can associate a \textit{Young diagram}: a left-justified array of boxes with $\lambda_i$ boxes in the $i$th row from the top. We often equate a partition $\lambda$ with its Young digram. The Young diagrams for the partitions of 4 are shown below.

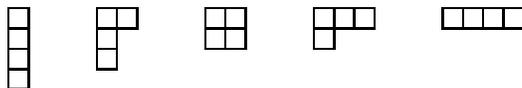
\begin{center}
\ytableausetup{boxsize=.1in}
\ydiagram{1,1,1,1}\hspace{.3in}
\ydiagram{2,1,1}\hspace{.3in}
\ydiagram{2,2}\hspace{.3in}
\ydiagram{3,1}\hspace{.3in}
\ydiagram{4}
\end{center}

A \textit{semistandard Young tableau} of shape $\lambda$ is a filling of the boxes of the Young diagram of shape $\lambda$ with positive integers such that the entries weakly increase from left to right along rows and strictly increase down columns. 
A \textit{standard Young tableau} of shape $\lambda\vdash n$ is a semistandard Young tableau of shape $\lambda$ such that each positive integer $1,2,\ldots,n$ appears exactly once. For a semistandard Young tableau $T$, let $sh(T)$ denote the shape of $T$ and $|T|$ denote the number of boxes of $T$ or, equivalently, the number of entries in $T$. For example, $T_1$ below is a semistandard Young tableau and $T_2$ is a standard Young tableau. We see that $|T_1|=|T_2|=11$ and $sh(T_1)=sh(T_2)=(5,3,3)$.
\begin{center}
\ytableausetup{boxsize=.2in}
$T_1=$
\begin{ytableau}
1 & 1 & 1 & 3 & 4 \\
3 & 3 & 5 \\
4 & 5 & 7
\end{ytableau}\hspace{.5in}
$T_2=$
\begin{ytableau}
1 & 3 & 4 & 9 & 11\\
2 & 5 & 7\\
6 & 8 & 10
\end{ytableau}
\end{center}

Consider a semistandard Young tableau as sitting in the southeast quadrant of the plane with top left corner at the origin and each box of side length 1. Notice that each cell of the tableau is crossed by a unique diagonal $D_k$, where $D_k$ is the line $-x+k$ for some $k\in\mathbb{Z}$. For example, the boxes of $T_2$ with entries 1, 5, and 10 lie on $D_0$ while the boxes with entries 2 and 8 lie on $D_{-1}$. The \textit{diagonal reading word} of a semistandard tableau $T$ is the word obtained by reading the entries along each diagonal from northwest to southeast starting with the bottom diagonal. We insert the symbol ``$/$'' between the segments obtained from each diagonal. For example, the diagonal reading word for $T_1$ and $T_2$ are respectively $4\ /\ 3, 5\ /\ 1, 3, 7\ /\ 1, 5 \ /\ 1\ /\ 3\ / \ 4$ and $6\ / \ 2,8\ / \ 1,5,10 \ / \ 3,7 \ / \ 4 \ / \ 9 \ / \ 11$. The diagonal reading word of a semistandard Young tableau defines it uniquely, a property that we will use in Theorem~\ref{thm:bijectiondomino}.

\subsection{Symmetric functions}\label{sec:symmetricfunctions}
A \textit{weak composition} is a countable sequence of non-negative integers $\alpha=(\alpha_1,\alpha_2,\ldots)$ such that only finitely many $\alpha_i$ are nonzero. Let $S_n$ denote the symmetric group of order $n!$, the group of all permutations of the set $\{ 1,2,\ldots, n \}$.

Let $x=(x_1,x_2,\ldots)$ be a countable set of variables, and for a weak composition $\alpha$, define $x^\alpha$ to be $x_1^{\alpha_1}x_2^{\alpha_2}\cdots$. A \textit{symmetric function} is a formal power series $f(x)=\sum_{\alpha}c_\alpha x^\alpha$ such that $c_\alpha\in\mathbb{R}$ and such that, for any non-negative integer $n$ and any $\sigma \in S_n$,
\[ f(x_{\sigma (1)}, x_{\sigma (2)},\ldots,x_{\sigma (n)}, x_{n+1},\ldots) = f(x_1, x_2,\ldots).\] 
We say that a symmetric function is \textit{homogeneous of degree $n$} if each of its monomials has degree $n$. We denote the set of homogeneous symmetric functions of degree $n$ by $\Lambda ^{n}$ and the set of symmetric functions by $\Lambda$.

For example, 
\[f(x)=\sum_{i\leq j} x_ix_j = x_1^2+x_1x_2+x_1x_3+\ldots \in \Lambda^2\] is a homogeneous symmetric function of degree 2,
\[g(x)=x_1+x_2+\ldots + x_1^2+x_2^2+\ldots\in \Lambda\] is a symmetric function but is not homogeneous of any degree, and 
\[h(x)=x_1+\sum_{i\in\mathbb{N}}x_i^2\] is not a symmetric function.

It is easy to see that $\Lambda$ is an algebra with identity element $1\in\Lambda^0$. In other words, $\Lambda$ is an $\mathbb{R}$-vector space under addition and a ring under multiplication. 

\subsection{Schur functions}\label{sec:Schurfunctions}
The algebra of symmetric functions has many nice bases, which are well studied. We next introduce one such basis: the basis of Schur functions. This basis is of great interest because of its connections to other areas of mathematics. For example, Schur functions are closely related to the irreducible representations of both the symmetric group and the general linear group. They also appear in the area of Schubert calculus as a tool for computing the structure constants in the cohomology ring of the Grassmannian. There are many ways to define the Schur functions and we use the combinatorial definition. 

Let $T$ be a semistandard Young tableau. We can associate a monomial $x^T$ in the variable set $(x_1,x_2,\ldots)$ to $T$ by letting the exponent of $x_i$ be the number of times the entry $i$ appears in $T$. For example, $x^{T_1}=x_1^3x_3^3x_4^2x_5^2x_7$ and $x^{T_2}=x_1x_2x_3x_4x_5x_6x_7x_8x_9x_{10}x_{11}$ for $T_1$ and $T_2$ from Section~\ref{sec:tableaux}.

The basis of Schur functions is indexed by partitions. We define the \textit{Schur function $s_\lambda$} by 
\[s_\lambda=s_\lambda(x) = \sum_{sh(T)=\lambda}x^T,\]
where we sum over all semistandard Young tableaux T of shape $\lambda$. Note that if $\lambda\vdash n$, then $s_\lambda \in \Lambda^n$. It is easy to see that each monomial has degree $n$, but it is not obvious from the combinatorial definition that $s_\lambda$ is indeed symmetric.

\begin{example}
We can compute that 
\[s_{(2,1)}=x_1^2x_2+x_1^2x_3+x_1^2x_9+x_1x_2^2+x_1x_2x_3+x_1x_2x_3+x_1x_3^2+x_2^2x_3+x_2x_3^2+\ldots,  \]
where the monomials given correspond to the semistandard tableaux below.

\begin{center}
\begin{ytableau}1 & 1 \\ 2 \end{ytableau}\hspace{.2in}
\begin{ytableau}1 & 1 \\ 3 \end{ytableau}\hspace{.2in}
\begin{ytableau}1 & 1 \\ 9 \end{ytableau}\hspace{.2in}
\begin{ytableau}1 & 2 \\ 2 \end{ytableau}\hspace{.2in}
\begin{ytableau}1 & 2 \\ 3 \end{ytableau}\hspace{.2in}
\begin{ytableau}1 & 3 \\ 2 \end{ytableau}\hspace{.2in}
\begin{ytableau}1 & 3 \\ 3 \end{ytableau}\hspace{.2in}
\begin{ytableau}2 & 2 \\ 3 \end{ytableau}\hspace{.2in}
\begin{ytableau}2 & 3 \\ 3 \end{ytableau}
\end{center}
\end{example}

Since the Schur functions form a linear basis for $\Lambda$, we know that we can express any product $s_\lambda s_\mu$ as a finite sum of Schur functions: $s_\lambda s_\mu =\sum_{\nu}c_{\lambda,\mu}^\nu s_\nu$. This idea has applications in representation theory and Schubert calculus, as mentioned above, and has been very well studied. The coefficients $c_{\lambda,\mu}^{\nu}$ are called \textit{Littlewood--Richardson coefficients}, and there are many combinatorial rules for computing them. In Theorem~\ref{thm:bijectiondominosym}, we give a rule for expressing this product as a sum over domino tableaux for certain pairs $s_\lambda$ and $s_\mu$.

\section{Domino tableaux}\label{sec:dominotableaux}

First, define a \textit{domino} to be a $2\times 1$ or $1\times 2$ rectangle inside of a Young diagram. The red and blue shapes below are both dominoes.
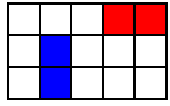
\begin{center}
\ytableausetup{boxsize=.16in}
\begin{ytableau}
$ $ & & & *(red) & *(red)\\
$ $ &*(blue) & & & \\
$ $ & *(blue) & & & 
\end{ytableau}
\end{center}
We say that a Young diagram is \textit{pavable} if it can be written as the disjoint union of dominoes. The reader may verify that the partition $(2,2,2)$ is pavable, while the partition $(5,5,5)$ shown above is not (it has an odd number of boxes) and the partition $(5,3,3,2,1)$ is not. If $\lambda$ is pavable, we call any such covering a \textit{domino paving}.

\begin{definition}\label{def:dominotab}
A \textit{domino tableau of shape $\lambda$} is the filling of a domino paving of $\lambda$ by positive integers such that:
\begin{itemize}
\item entries weakly increase from left to right and
\item columns strictly increase from top to bottom.
\end{itemize}
\end{definition}

We again think of the top left corner of a domino tableau as sitting at the origin of the plane. In this setting, each domino in a domino tableau is crossed by a unique diagonal $D_{2k}$ of equation $ -x+2k$.
We define the \textit{diagonal reading word} of a domino tableau to be the integer sequence obtained by reading northwest to southeast along each diagonal $D_{2k}$ starting with the bottom diagonal. We again separate the entries on distinct diagonals by ``/''. Unlike for the diagonal reading of semistandard Young tableaux, the diagonal reading of a domino tableau does not define it uniquely. For example, the diagonal reading ``1'' could refer to a single vertical domino or a single horizontal domino. 

\begin{example}
The following figure represents a domino tableau of shape $(5,4,2,1)$ with its diagonals. The diagonal reading word of this tableau is  2 / 1,3 / 1,6 / 5.  
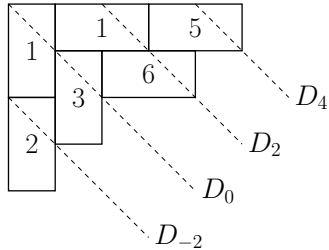
\begin{center}
\resizebox{4.5cm}{!}{
\begin{tikzpicture}[node distance=0 cm,outer sep = 0pt]
        \tikzstyle{every node}=[font=\LARGE]
        \node[bver] (1) at ( 1.5, 3) {1};
        \node[bver] (2) [below = of 1] {2};
        \node[bhor] (3) at ( 3, 3.5) {1};
        \node[bver] (4) at ( 2.5, 2) {3};
        \node[bhor] (5) [right = of 3] {5};
        \node[bhor] (6) at ( 4, 2.5) {6};
        \draw [dashed] ( 1, 4) -- ( 5, 0);
        \draw ( 5, 0) node[right] {$D_0$};
        \draw [dashed] ( 1, 2) -- ( 4, -1);
        \draw ( 4, -1) node[right] {$D_{-2}$};
        \draw [dashed] ( 3, 4) -- ( 6, 1);
        \draw ( 6, 1) node[right] {$D_{2}$};
        \draw [dashed] ( 5, 4) -- ( 7, 2);
        \draw ( 7, 2) node[right] {$D_{4}$};
\end{tikzpicture}
}
\end{center}
\end{example}

We can divide the dominoes of a domino paving into two categories depending on how they are cut by a diagonal $D_{2k}$:

\begin{enumerate}
\item We call a domino a \textit{type 1 domino} if the small triangle cut by the diagonal points upward. 

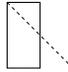
\begin{minipage}{0.4\linewidth}
\begin{center}
\resizebox{1cm}{!}{
\begin{tikzpicture}[node distance=0 cm,outer sep = 0pt]
        \node[bver] (1) at ( 1.5, 3) {};
        \draw [dashed, thick] ( 1, 4) -- ( 3, 2);
\end{tikzpicture}
}
\end{center}
\end{minipage}
\quad
\begin{minipage}{0.4\linewidth}
\begin{center}
\resizebox{1.2cm}{!}
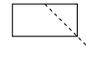
{
\begin{tikzpicture}[node distance=0 cm,outer sep = 0pt]
        \node[bhor] (1) at ( 1.5, 3) {};
        \draw [dashed, thick] ( 1.5, 3.5) -- ( 3, 2);
\end{tikzpicture}
}
\end{center}
\end{minipage}

\item We call a domino a \textit{type 2 domino} if the small triangle cut by the diagonal points downward. 

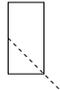
\begin{minipage}{0.4\linewidth}
\begin{center}
\resizebox{0.85cm}{!}{
\begin{tikzpicture}[node distance=0 cm,outer sep = 0pt]
        \node[bver] (1) at ( 1.5, 3) {};
        \draw [dashed, thick] ( 1, 3) -- ( 2.5, 1.5);
\end{tikzpicture}
}
\end{center}
\end{minipage}
\quad
\begin{minipage}{0.4\linewidth}
\begin{center}
\resizebox{1cm}{!}
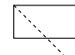
{
\begin{tikzpicture}[node distance=0 cm,outer sep = 0pt]
        \node[bhor] (1) at ( 1.5, 3) {};
        \draw [dashed, thick] ( 0.5, 3.5) -- ( 2, 2);
\end{tikzpicture}
}
\end{center}
\end{minipage}
\end{enumerate}

We next define the \textit{2-quotient} of a partition $\lambda=(\lambda_1,\ldots,\lambda_k)\vdash n$, a pair of partitions $(\mu,\nu)$ obtained in the following way:

\begin{enumerate}
\item First define $L= (l_1, l_2,\ldots, l_k )$, where $l_i = \lambda_i + k - i$ for $i\in \{ 1,2,...,k \}$.
\item Let $M$ be obtained from $L$ by successively replacing the even components of $L$ by $0,2,4,\ldots$ from right to left and the odd components by $1,3,5\ldots$ from right to left.
\item To obtain $\mu$, subtract the even components of $L$ by the even components of $M$ and divide by 2. Delete the components that are 0.
\item To obtain $\nu$, subtract the odd components of $L$ by the odd components of $M$ and divide by 2. Delete the components that are 0.
\end{enumerate}
\begin{example}
Let $\lambda = (4,2,2,1,1,1)$. Then we have that 

\begin{enumerate}
\item $L = (4+6-1, 2+6-2, 2+6-3, 1+6-4, 1+6-5, 1+6-6) = (9,6,5,3,2,1)$,
\item $M = (7,2,5,3,0,1)$,
\item $\mu = \frac{1}{2} ( (6,2) - (2,0) ) = \frac{1}{2}  (4,2) = (2,1)$, and
\item $\nu = \frac{1}{2} ( (9,5,3,1) - (7,5,3,1) ) = \frac{1}{2} (2,0,0,0) = (1,0,0,0)=(1)$.
\end{enumerate}
Thus the 2-quotient of $(4,2,2,1,1,1)$ is the pair $((2,1),(1))$.
\end{example}
Note that this process is reversible, i.e., every pair of partitions $(\mu,\nu)$ is the 2-quotient of some partition $\lambda$. We discuss the reverse procedure in the next section.


\section{Bijection between domino tableaux and semistandard tableaux}\label{sec:bijection}
We now describe the bijection used to prove the following theorem. Our main result of Section~\ref{sec:K-theory} is a generalization of this theorem, so it will be useful in later sections to understand this bijection.

\begin{theorem}\cite{carre1995splitting,stanton1985schensted}\label{thm:bijectiondomino}
Let $\lambda$ be a pavable partition with 2-quotient ($\mu, \nu$). There is a bijection between the set of domino tableaux of shape $\lambda$ and the set of pairs of Young tableaux $(t_1,t_2)$ of shape $(\mu, \nu)$.
\end{theorem}

Theorem ~\ref{thm:bijectiondomino} is proven by giving an explicit bijection $\Gamma$ that sends a domino tableau to the associated pair of Young tableaux. The bijection $\Gamma$ consists of considering the diagonal reading of entries in type 1 dominoes and of type 2 dominoes separately. More precisely, let $T$ be a domino tableau, and form the diagonal reading word for $T$. Let $w_1$ be the word obtained by restricting this diagonal reading word to the entries that come from type 1 dominoes and let $w_2$ be the word obtained by restricting the diagonal reading word for $T$ to entries coming from type 2 dominoes. We then let semistandard tableau $t_1$ be the unique Young tableau with diagonal reading word $w_1$ and let $t_2$ be the unique Young tableau with diagonal reading $w_2$. We illustrate this bijection below using an example.
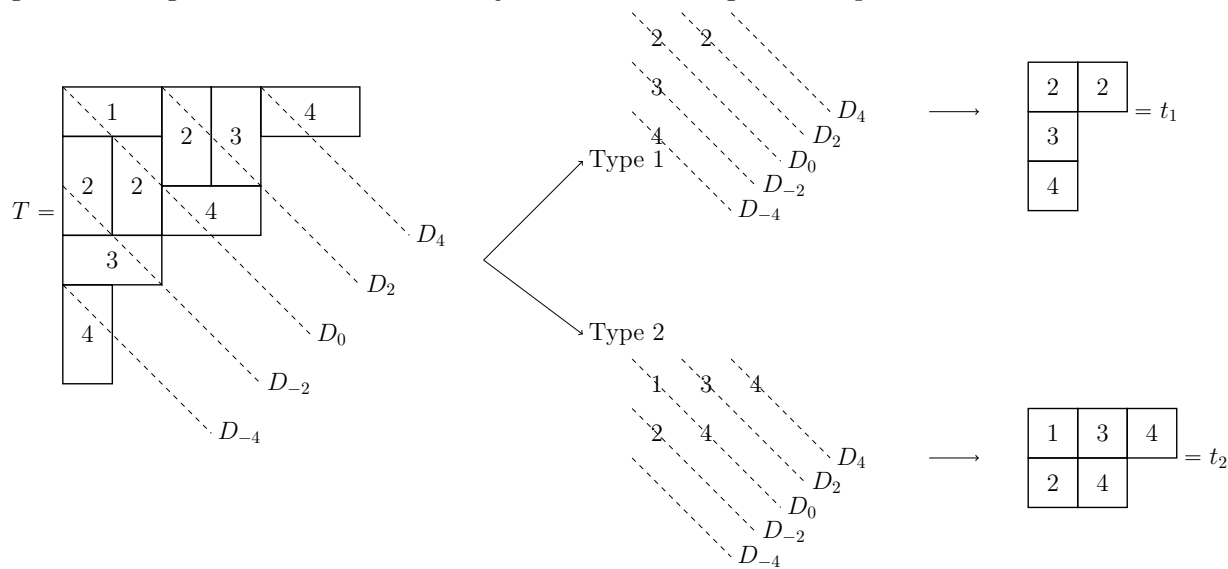
\begin{center}
\resizebox{16.5cm}{!}{
\begin{tikzpicture}[node distance=0 cm,outer sep = 0pt]

        \tikzstyle{every node}=[font=\Large]
        
        \draw ( 0, 6) node {$T = $ \ };

        \node[bhor] (1) at ( 1.5, 8) {1};
        \node[bver] (2) at ( 1, 6.5) {2};
        \node[bver] (3) [right = of 2] {2};
        \node[bver] (4) at ( 3, 7.5) {2};
        \node[bver] (5) [right = of 4] {3};
        \node[bhor] (6) at ( 5.5, 8) {4};
        \node[bhor] (7) at ( 3.5, 6) {4};
        \node[bhor] (8) at ( 1.5, 5) {3};
        \node[bver] (9) at ( 1, 3.5) {4};

        \draw [dashed] ( .5, 8.5) -- ( 5.5, 3.5);
        \draw ( 5.5, 3.5) node[right] {$D_0$};
        \draw [dashed] ( 2.5, 8.5) -- ( 6.5, 4.5);
        \draw ( 6.5, 4.5) node[right] {$D_{2}$};
        \draw [dashed] ( 4.5, 8.5) -- ( 7.5, 5.5);
        \draw ( 7.5, 5.5) node[right] {$D_{4}$};
        \draw [dashed] ( .5, 6.5) -- ( 4.5, 2.5);
        \draw ( 4.5, 2.5) node[right] {$D_{-2}$};
        \draw [dashed] ( .5, 4.5) -- ( 3.5, 1.5);
        \draw ( 3.5, 1.5) node[right] {$D_{-4}$};
        
        \draw [->] ( 9, 5) -- ( 11, 7);
        \draw (11,7) node[right] {Type 1};
        \draw [->] ( 9, 5) -- ( 11, 3.5);
        \draw (11,3.5) node[right] {Type 2};
        
        \draw [dashed] ( 12, 10) -- ( 15, 7);
        \draw ( 15, 7) node[right] {$D_0$};
        \draw [dashed] ( 13, 10) -- ( 15.5, 7.5);
        \draw ( 15.5, 7.5) node[right] {$D_{2}$};
        \draw [dashed] ( 14, 10) -- ( 16, 8);
        \draw ( 16, 8) node[right] {$D_{4}$};
        \draw [dashed] ( 12, 9) -- ( 14.5, 6.5);
        \draw ( 14.5, 6.5) node[right] {$D_{-2}$};
        \draw [dashed] ( 12, 8) -- ( 14, 6);
        \draw ( 14, 6) node[right] {$D_{-4}$};
        
        \draw [dashed] ( 12, 3) -- ( 15, 0);
        \draw ( 15, 0) node[right] {$D_0$};
        \draw [dashed] ( 13, 3) -- ( 15.5, .5);
        \draw ( 15.5, .5) node[right] {$D_{2}$};
        \draw [dashed] ( 14, 3) -- ( 16, 1);
        \draw ( 16, 1) node[right] {$D_{4}$};
        \draw [dashed] ( 12, 2) -- ( 14.5, -.5);
        \draw ( 14.5, -.5) node[right] {$D_{-2}$};
        \draw [dashed] ( 12, 1) -- ( 14, -1);
        \draw ( 14, -1) node[right] {$D_{-4}$};
        
        \draw ( 12.5, 9.5) node {2};
        \draw ( 13.5, 9.5) node {2};
        \draw ( 12.5, 8.5) node {3};
        \draw ( 12.5, 7.5) node {4};
        
        \draw ( 12.5, 2.5) node {1};
        \draw ( 13.5, 2.5) node {3};
        \draw ( 14.5, 2.5) node {4};
        \draw ( 12.5, 1.5) node {2};
        \draw ( 13.5, 1.5) node {4};
        
        \draw [->] ( 18, 8) -- ( 19, 8);
        \draw [->] ( 18, 1) -- ( 19, 1);
        
        \node[bsq] (10)   at ( 20.5, 8.5) {2};
        \node[bsq] (11)  [right = of 10] {2};
        \node[bsq] (12)   at ( 20.5, 7.5) {3};
        \node[bsq] (13)  [below = of 12] {4};
        \draw ( 22.5, 8) node {\ $ = t_1$};
        
        \node[bsq] (14)   at ( 20.5, 1.5) {1};
        \node[bsq] (15)  [right = of 14] {3};
        \node[bsq] (16)  [right = of 15] {4};
        \node[bsq] (17)  [below = of 14] {2};
        \node[bsq] (18)  [right = of 17] {4};
        \draw ( 23.5, 1) node {\ $ = t_2$};

\end{tikzpicture}
}
\end{center}

We leave it to the reader to verify that the 2-quotient of $sh(T) = ( 6, 4, 4, 2, 1, 1)$ is $((2,1,1), (3,2))$, the shape of $( t_1, t_2)$.

The inverse algorithm, $\Gamma ^{-1}$, consists of recursively constructing the domino tableau of shape $\lambda$ associated to a pair of Young tableaux $(t_1, t_2)$ of shape $(\mu, \nu)$, where $(\mu, \nu)$ is the 2-quotient of $\lambda$. At any step, we have a pair of Young tableaux $(t_1^{(i)}, t_2^{(i)})$ of shape $(\mu^{(i)}, \nu^{(i)})$ and the associated domino tableau $T^{(i)}$ of shape $\lambda^{(i)}$. We start the algorithm with $\mu^{(0)} = \nu^{(0)} = \lambda^{(0)} = \varnothing$. The algorithm stops when $(t_1^{(s)}, t_2^{(s)}) = (t_1, t_2)$. Then we have that the domino tableau associated to $(t_1, t_2)$ is $T^{(s)}$. We now describe the $i ^{th}$ step of the algorithm. 

Let $u_i$ be the smallest value appearing in $(t_1, t_2)$ that does not appear in $(t_1^{(i-1)}, t_2^{(i-1)})$. 
We build $(t_1^{(i)}, t_2^{(i)})$ of shape $(\mu ^{(i)}, \nu ^{(i)})$ by adding to $(t_1^{(i-1)}, t_2^{(i-1)})$ all cells of $(t_1, t_2)$ with value $u_i$, while preserving their original position. 

To construct the domino tableau $T^{(i)}$ of shape $\lambda^{(i)}$, we use the following procedure for each diagonal, starting with the bottom diagonal: For all cells in $t_1^{(i)}$ (resp. $t_2^{(i)}$) containing the value $u_i$ on diagonal $D_k$, we add to $T^{(i-1)}$ a type 1 domino (resp. a type 2 domino) with entry  $u_i$ on the corresponding diagonal $D_{2k}$. We then get the associated domino tableau $T^{(i)}$ of shape $\lambda^{(i)}$.


Below is an example of $\Gamma^{-1}$ applied to the pair of Young tableaux \[(t_1, t_2) = \left( \raisebox{.15in}{ \ \begin{ytableau} 2 & 2 \\ 3 \\ 4 \end{ytableau}, \ \begin{ytableau} 1 & 3 & 4 \\ 2 & 4 \end{ytableau} \ } \right).\] Notice that we recover the tableau $T$ from the previous example.

\begin{minipage}{.5\textwidth}
\raisebox{.1in}{$(1)$ $\left( \ \varnothing, \ \begin{ytableau}1 \end{ytableau} \ \right) \rightarrow $ }
\resizebox{1cm}{!}{
\begin{tikzpicture}[node distance=0 cm,outer sep = 0pt]
        \tikzstyle{every node}=[font=\LARGE]
        \node[bhor] (1) at ( 1.5, 3) {1};
        
\end{tikzpicture}
}\end{minipage}\begin{minipage}{.5\textwidth}
\raisebox{.2in}{$(2)$ $ \left( \ \begin{ytableau} 2 & 2 \end{ytableau}, \ \begin{ytableau} 1 \\ 2 \end{ytableau} \ \right) \rightarrow $ }
\resizebox{1.5cm}{!}{
\begin{tikzpicture}[node distance=0 cm,outer sep = 0pt]
        \tikzstyle{every node}=[font=\LARGE]
        \node[bhor] (1) at ( 1.5, 1) {1};
        \node[bver] (2) at ( 1, -.5) {2};
        \node[bver] (3) at ( 2, -.5) {2};
        \node[bver] (4) at ( 3, .5) {2};

\end{tikzpicture}
} \end{minipage} 

\begin{minipage}{.5\textwidth}
\raisebox{.4in}{$(3)$ $ \left( \ \begin{ytableau} 2 & 2 \\ 3 \end{ytableau}, \ \begin{ytableau} 1 & 3 \\ 2 \end{ytableau} \ \right) \rightarrow $ }
\resizebox{2cm}{!}{
\begin{tikzpicture}[node distance=0 cm,outer sep = 0pt]
        \tikzstyle{every node}=[font=\LARGE]
        \node[bhor] (1) at ( 1.5, 3) {1};
        \node[bver] (2) at ( 1, 1.5) {2};
        \node[bver] (3) at ( 2, 1.5) {2};
        \node[bver] (4) at ( 3, 2.5) {2};
        \node[bhor] (5) at ( 1.5, 0) {3};
        \node[bver] (6) at ( 4, 2.5) {3};
\end{tikzpicture}
}\end{minipage}\begin{minipage}{.5\textwidth}
\raisebox{.6in}{$(4)$ $ \left( \raisebox{.1in}{ \ \begin{ytableau} 2 & 2 \\ 3 \\ 4 \end{ytableau}, \ \begin{ytableau} 1 & 3 & 4 \\ 2 & 4 \end{ytableau} \ } \right) \rightarrow $ }
\resizebox{3cm}{!}{
\begin{tikzpicture}[node distance=0 cm,outer sep = 0pt]
        \tikzstyle{every node}=[font=\LARGE]
        \node[bhor] (1) at ( 1.5, 3) {1};
        \node[bver] (2) at ( 1, 1.5) {2};
        \node[bver] (3) at ( 2, 1.5) {2};
        \node[bver] (4) at ( 3, 2.5) {2};
        \node[bhor] (5) at ( 1.5, 0) {3};
        \node[bver] (6) at ( 4, 2.5) {3};
        \node[bhor] (7) at ( 5.5, 3) {4};
        \node[bhor] (8) at ( 3.5, 1) {4};
        \node[bver] (9) at ( 1,-1.5) {4};
\end{tikzpicture}
}\end{minipage}

The following is a corollary of Theorem~\ref{thm:bijectiondomino}. It is important to note that this result holds for any pair of partitions $(\mu,\nu)$ because we can reverse the 2-quotient procedure, as illustrated in the example of $\Gamma^{-1}$ above.

\begin{theorem}\label{thm:bijectiondominosym}
Let $\lambda $ be a partition with 2-quotient ($\mu$, $\nu$). Then
\begin{equation*}
s_{\mu} s_{\nu} = \sum _{T} x^T
\end{equation*}
where $T$ runs over the set of domino tableaux of shape $\lambda$.
\end{theorem}
\begin{proof}
Each term in the product $s_\mu s_\nu$ is represented by a pair of semistandard Young tableaux of shape $(\mu,\nu)$. Theorem~\ref{thm:bijectiondomino} says that these pairs are in bijection with domino tableaux of shape $\lambda$, and the domino tableau associated to a pair $(t_1,t_2)$ has the same multiset of entries as $(t_1,t_2)$.
\end{proof}

\section{$K$-theoretic generalizations}\label{sec:K-theory}
\subsection{Stable Grothendieck polynomials}\label{sec:Grothendieck}
$K$-theory is a generalized cohomology theory. As mentioned in Section~\ref{sec:Schurfunctions}, the Schur functions are deeply connected to the cohomology of the Grassmannian,  and play a very specific role in that cohomology. It turns out that there is a generalization of the Schur functions that plays the same role in the $K$-theory of the Grassmannian; these symmetric functions are called the stable Grothendieck polynomials and are denoted by $G_\lambda$, where $\lambda$ is a partition. We will only discuss the combinatorics of $K$-theory here, and it is not necessary for the reader to have any prior knowledge of the geometry.

Stable Grothendieck polynomials were introduced by Fomin and Kirillov \cite{fomin1996yang} as certain limits of Lascoux and Schutzenberger's Grothendieck polynomials \cite{lascoux1982}. We will give the combinatorial definition first explicitly written in \cite{buch2002littlewood}. The rough idea is that $G_\lambda$ is defined in the same way as $s_\lambda$ except that we are allowed to fill the boxes of $\lambda$ with finite subsets of integers  instead of just single integers. To this end, we first define a partial ordering on subsets.

Let $A$ and $B$ be finite, nonempty sets of positive integers. We say that $A\leq B$ if $\max(A) \leq \min(B)$ and $A< B$ if $\max(A) < \min(B)$. For example, $\{1, 3, 4\}\leq \{4,6,7,33\}$, $\{1,3,4\}<\{5\}$, and $\{1,3,4\}$ is not comparable to $\{2,5,7\}$.

\begin{definition}
Let $\lambda$ be a partition. A \textit{semistandard set-valued tableau of shape $\lambda$} is a filling of the Young diagram $\lambda$ with finite, non-empty sets of positive integers such that
\begin{itemize}
\item entries are weakly increasing from left to right along the rows and
\item entries are strictly increasing down columns.
\end{itemize}
\end{definition}

Given a semistandard set-valued tableau $T$, we may again associate a monomial $x^T$ where 
\[x^T = x_1 ^{\alpha_1 (T)} x_2 ^{\alpha_2 (T)} x_3 ^{\alpha_3 (T)} \cdots,\]
where $\alpha_i (T)$ is the number of occurrences of $i$ in $T$. We let $|T|$ denote the sum of the sizes of the sets that fill $T$ or, equivalently, the total number of integers filling $T$. To illustrate, the tableau $T$ shown below has $x^T=x_1 x_3 ^2 x_4 x_6 x_7 x_9
$ and $|T|=8$. 
\begin{center}
$T =$ 
\ytableausetup{boxsize=.3in}
\begin{ytableau}1,3 & 3 & 6,7 \\ 4 & 5,9
\end{ytableau}
\end{center}
Note that if we pick one representative from each box of a semistandard set-valued tableau of shape $\lambda$, we obtain a semistandard Young tableau of shape $\lambda$.

We can again define the \textit{diagonal reading word} for a set-valued tableau $T$ in a similar way. We will read the entries northwest to southeast along each diagonal, starting with the bottom diagonal. The only difference with the diagonal reading of a semistandard Young tableau is that we use braces to identify the elements of a set that has more than one element. For example, the diagonal reading word for tableau $T$ above is $4\ /\ \{1,3\}, \{5,9\}
\ /\ 3\ /\ \{6,7\}$. Just like for semistandard tableaux, the diagonal reading of a set-valued tableau defines it uniquely.

\begin{definition}\cite{buch2002littlewood} Let $\lambda\vdash n$. The \textit{stable Grothendieck polynomial} $G_\lambda$ is 
\[G_\lambda = \displaystyle\sum_{sh(T)=\lambda}(-1)^{|T|-|\lambda|} x^T,\] where we sum over all semistandard set-valued tableaux $T$ of shape $\lambda$.
\end{definition}

\begin{example}
We have that 
\[G_{(2,1)} = x_1^2x_2+2x_1x_2x_3 - x_1 ^2 x_2 ^2 - 3x_1 ^2 x_2 x_3 + 3x_2 x_6 x_7 ^2 x_8 \pm \ldots,\] where the terms shown correspond to the tableaux below. Note that there are additional tableaux with monomial $x_2 x_6 x_7 ^2 x_8$, so the coefficient of $x_2 x_6 x_7 ^2 x_8$ in the full $G_{(2,1)}$ is greater than 3.

\vspace{.15in}
\begin{center}
\ytableausetup{boxsize=.25in}
\begin{ytableau}\scriptstyle{1} & \scriptstyle{1} \\ \scriptstyle{2}\end{ytableau}
\hspace{.06in}
\begin{ytableau}\scriptstyle{1} & \scriptstyle{3} \\ \scriptstyle{2}\end{ytableau}
\hspace{.06in}
\begin{ytableau}\scriptstyle{1} & \scriptstyle{2} \\ \scriptstyle{3}\end{ytableau}
\hspace{.06in}
\begin{ytableau}\scriptstyle{1} & \scriptstyle{1,2} \\ \scriptstyle{2} \end{ytableau} \hspace{.06in}
\begin{ytableau}\scriptstyle{1} & \scriptstyle{1,2} \\ \scriptstyle{3} \end{ytableau} \hspace{.06in}
\begin{ytableau}\scriptstyle{1} & \scriptstyle{1,3} \\ \scriptstyle{2} \end{ytableau} \hspace{.06in}
\begin{ytableau}\scriptstyle{1} & \scriptstyle{1} \\ \scriptstyle{2,3} \end{ytableau} \hspace{.06in}
\begin{ytableau}\scriptstyle{2,6} & \scriptstyle{7} \\ \scriptstyle{7,8} \end{ytableau} \hspace{.06in}
\begin{ytableau}\scriptstyle{2,6} & \scriptstyle{7,8} \\ \scriptstyle{7} \end{ytableau} \hspace{.06in}
\begin{ytableau}\scriptstyle{2} & \scriptstyle{6,7} \\ \scriptstyle{7,8} \end{ytableau} \hspace{.06in}
\end{center}
\vspace{.15in}

Notice that since the set of semistandard Young tableaux is contained in the set of semistandard set-valued tableaux, $G_\lambda$ will contain $s_\lambda$ as the set of terms of lowest degree. This observation shows us that $G_\lambda$ is a generalization of $s_\lambda$, and, in particular, any formula we have for expressing a product $G_\mu G_\nu$ in terms of stable Grothendieck polynomials will restrict to a formula for $s_\mu s_\nu$ upon restriction to the lowest degree terms. Notice also that the monomials of $G_\lambda$ have arbitrarily large degree. 
\end{example}

\subsection{Set-valued domino tableaux}
In this section, we introduce the notion of a set-valued domino tableau. We will use this object to prove a $K$-theoretic analogue of Theorem~\ref{thm:bijectiondomino}.

If $F$ is a domino filled with subset $A$, let $\max(F)$ and $\min(F)$ denote $\max(A)$ and $\min(A)$, respectively. We say a (square) box of a partition is in position $(i,j)$ if it is in the $i$th row and $j$th column of that partition. Suppose the top left square of domino $F_1$ is in position $(i,j)$ of pavable partition $\lambda$. We say that domino $F_2$ is \textit{weakly southeast} of domino $F_1$ if $F_2$ intersects a box of $\lambda$ in position $(k,\ell)$ with $k\geq i$ and $\ell\geq j$. In other words, at least part of $F_2$ is weakly southeast of the top left square of $F_1$. 


\begin{definition}
A \textit{set-valued domino tableau of shape $\lambda$} is the filling of a domino paving of $\lambda$ with finite, non-empty sets of positive integers such that: 
\begin{enumerate}
\item Restricting to the minimum entry in each domino yields a domino tableau.

\item If $F_1$ and $F_2$ are dominoes of the same type on neighboring diagonals and $F_2$ is weakly southeast of $F_1$, then
\begin{itemize}
\item $\max(F_1)\leq\min(F_2)$ if $F_1$ is located on $D_{2k}$ and $F_2$ is on $D_{2(k+1)}$
\item $\max(F_1) < \min(F_2)$ if $F_1$ is located on $D_{2(k+1)}$ and $F_2$ is on $D_{2k}$.
\end{itemize}
\end{enumerate}
\end{definition}

Another way to state the first condition is that the minimum entries in the dominoes weakly increase from left to right along rows and strictly increase down columns.
 
For $T$ a set-valued domino tableau, we again let $|T|$ denote the total number of positive integers in the filling. We can also define a diagonal reading word  to be the sequence obtained by reading northwest to southeast along each diagonal $D_{2k}$, starting with the bottom diagonal, and separating entries on distinct diagonals with ``/''. We use braces to identify the elements of a set that contains more than one element.

\begin{example}
A set-valued domino tableau of shape $\lambda = (6,5,5,3,1)$ is shown below. Notice, for example, that the entry $\{3,4\}$ appears to the left of the entry $\{3\}$. This is acceptable because the corresponding dominoes have different types, and we therefore must only check that the minimum entries are weakly increasing across rows and strictly increasing down columns.

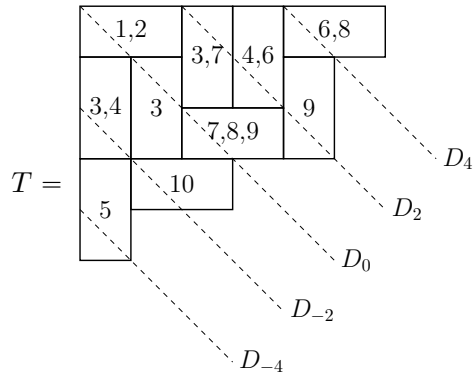
\begin{center} 
\raisebox{2.5cm}{$T =$} 
\resizebox{5.5cm}{!}{
\begin{tikzpicture}[node distance=0 cm,outer sep = 0pt]
        \tikzstyle{every node}=[font=\Large]
        \node[bhor] (1) at ( 1,   3.5) {1,2};
        \node[bver] (2) at (0.5,    2) {3,4};
        \node[bver] (3) [right = of 2] {3};
        \node[bver] (4) at ( 2.5,   3) {3,7};
        \node[bver] (4) at ( 3.5,   3) {4,6};
        \node[bhor] (5) at (  5,  3.5) {6,8};
        \node[bver] (6) at ( 4.5,   2) {9};
        \node[bhor] (7) at ( 3,   1.5) {7,8,9};
        \node[bhor] (8) at ( 2,   0.5) {10};
        \node[bver] (9) at ( 0.5,   0) {5};
        
        \draw [dashed] ( 0, 4) -- ( 5, -1);
        \draw ( 5, -1) node[right] {$D_0$};
        \draw [dashed] ( 2, 4) -- ( 6, 0);
        \draw ( 6, 0) node[right] {$D_{2}$};
        \draw [dashed] ( 4, 4) -- ( 7, 1);
        \draw ( 7, 1) node[right] {$D_{4}$};
        \draw [dashed] ( 0, 2) -- ( 4, -2);
        \draw ( 4, -2) node[right] {$D_{-2}$};
        \draw [dashed] ( 0, 0) -- ( 3, -3);
        \draw ( 3, -3) node[right] {$D_{-4}$};
\end{tikzpicture}
} 
\end{center}
Its diagonal reading word is 5 / \{3,4\},10 / \{1,2\},3,\{7,8,9\} / \{3,7\},\{4,6\},9 / \{6,8\} and $|T| = 17$.
\end{example}

We may now prove our main result of this section.

\begin{theorem}\label{thm:bijectiondomino2}
Let $\lambda$ be a pavable partition with 2-quotient ($\mu, \nu$). There is a bijection between the set of set-valued domino tableaux of shape $\lambda$ and the set of pairs of semistandard set-valued tableaux of shape $(\mu, \nu)$.
\end{theorem}

\begin{proof}
To prove this theorem, we generalize the maps $\Gamma$ and $\Gamma^{-1}$ from the proof of Theorem~\ref{thm:bijectiondomino}. We denote our generalized bijection by $\Gamma^*$ and describe how it maps a set-valued domino tableau of shape $\lambda$ to a pair of 
semistandard set-valued tableaux $(t_1,t_2)$ of shape $(\mu, \nu)$, where $(\mu, \nu)$ is the 2-quotient of $\lambda$. 

Let $T$ be a set-valued domino tableau of shape $\lambda$ and form the diagonal reading word for $T$. Let $w_1$ be the word obtained by restricting this diagonal reading word to the entries that come from type 1 dominoes and $w_2$ be the word obtained by restricting the diagonal reading word for $T$ to entries coming from type 2 dominoes. We then construct $t_1$ to be the unique set-valued tableau with diagonal reading word $w_1$ and $t_2$ to be the unique set-valued tableau with diagonal reading word $w_2$. We illustrate this bijection below using an example.

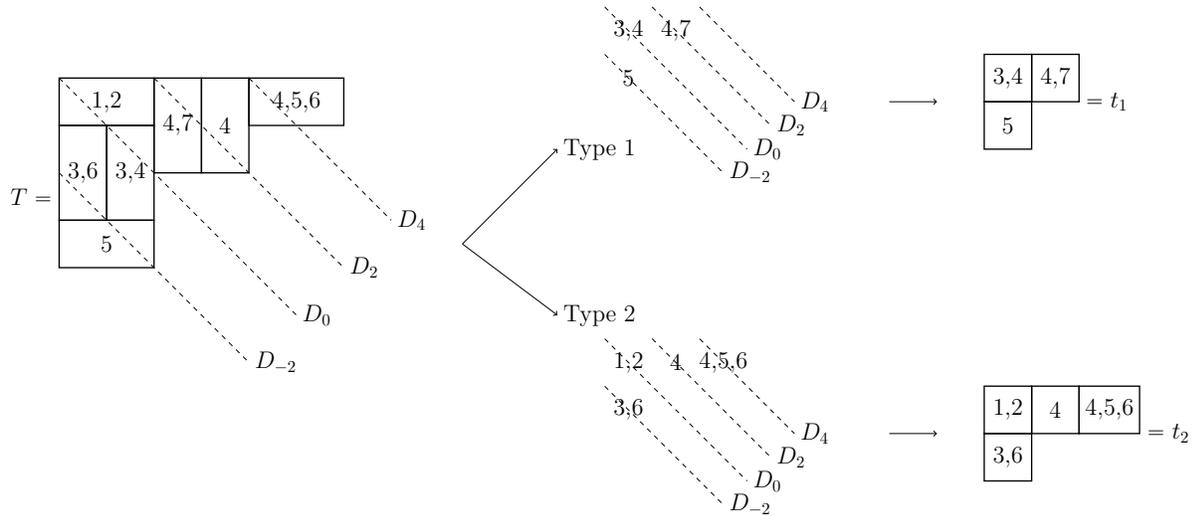
\begin{center}
\resizebox{16cm}{!}{
\begin{tikzpicture}[node distance=0 cm,outer sep = 0pt]

        \tikzstyle{every node}=[font=\Large]
        
        \draw ( 0, 6) node {$T = $ \ };

        \node[bhor] (1) at ( 1.5, 8) {1,2};
        \node[bver] (2) at ( 1, 6.5) {3,6};
        \node[bver] (3) [right = of 2] {3,4};
        \node[bver] (4) at ( 3, 7.5) {4,7};
        \node[bver] (5) [right = of 4] {4};
        \node[bhor] (6) at ( 5.5, 8) {4,5,6};
        \node[bhor] (8) at ( 1.5, 5) {5};

        \draw [dashed] ( .5, 8.5) -- ( 5.5, 3.5);
        \draw ( 5.5, 3.5) node[right] {$D_0$};
        \draw [dashed] ( 2.5, 8.5) -- ( 6.5, 4.5);
        \draw ( 6.5, 4.5) node[right] {$D_{2}$};
        \draw [dashed] ( 4.5, 8.5) -- ( 7.5, 5.5);
        \draw ( 7.5, 5.5) node[right] {$D_{4}$};
        \draw [dashed] ( .5, 6.5) -- ( 4.5, 2.5);
        \draw ( 4.5, 2.5) node[right] {$D_{-2}$};
        
        \draw [->] ( 9, 5) -- ( 11, 7);
        \draw (11,7) node[right] {Type 1};
        \draw [->] ( 9, 5) -- ( 11, 3.5);
        \draw (11,3.5) node[right] {Type 2};
        
        \draw [dashed] ( 12, 10) -- ( 15, 7);
        \draw ( 15, 7) node[right] {$D_0$};
        \draw [dashed] ( 13, 10) -- ( 15.5, 7.5);
        \draw ( 15.5, 7.5) node[right] {$D_{2}$};
        \draw [dashed] ( 14, 10) -- ( 16, 8);
        \draw ( 16, 8) node[right] {$D_{4}$};
        \draw [dashed] ( 12, 9) -- ( 14.5, 6.5);
        \draw ( 14.5, 6.5) node[right] {$D_{-2}$};
        
        \draw [dashed] ( 12, 3) -- ( 15, 0);
        \draw ( 15, 0) node[right] {$D_0$};
        \draw [dashed] ( 13, 3) -- ( 15.5, .5);
        \draw ( 15.5, .5) node[right] {$D_{2}$};
        \draw [dashed] ( 14, 3) -- ( 16, 1);
        \draw ( 16, 1) node[right] {$D_{4}$};
        \draw [dashed] ( 12, 2) -- ( 14.5, -.5);
        \draw ( 14.5, -.5) node[right] {$D_{-2}$};
        
        \draw ( 12.5, 9.5) node {3,4};
        \draw ( 13.5, 9.5) node {4,7};
        \draw ( 12.5, 8.5) node {5};
        
        \draw ( 12.5, 2.5) node {1,2};
        \draw ( 13.5, 2.5) node {4};
        \draw ( 14.5, 2.5) node {4,5,6};
        \draw ( 12.5, 1.5) node {3,6};
        
        \draw [->] ( 18, 8) -- ( 19, 8);
        \draw [->] ( 18, 1) -- ( 19, 1);
        
        \node[bsq] (10)   at ( 20.5, 8.5) {3,4};
        \node[bsq] (11)  [right = of 10] {4,7};
        \node[bsq] (12)   at ( 20.5, 7.5) {5};
        \draw ( 22.5, 8) node {\ $ = t_1$};
        
        \node[bsq] (14)   at ( 20.5, 1.5) {1,2};
        \node[bsq] (15)  [right = of 14] {4};
        \node[bsq] (16)  [right = of 15] {4,5,6};
        \node[bsq] (17)  [below = of 14] {3,6};
        \draw ( 23.8, 1) node {\ $ = t_2$};

\end{tikzpicture}
}
\end{center}

Let $T$ be a set-valued domino tableau as before, and we show that $\Gamma^*(T)=(t_1,t_2)$ is a pair of semistandard set-valued tableaux. We first show that the sets in $t_1$ and $t_2$ weakly increase along rows. Suppose box $b_1$ lies directly left of box $b_2$ in $t_i$. Now let $F_1$  and $F_2$ be the dominoes of $T$ such that $\Gamma^*$ sends the entries of $F_1$ to $b_1$ and those of $F_2$ to $b_2$. Note then that $F_2$ lies on the diagonal to the right of that of $F_1$. We will show that $F_2$ is weakly southeast of $F_1$.

Since taking the smallest entry in each domino of $T$ gives a domino tableau and $\Gamma^*$ restricts to $\Gamma$ on domino tableaux, we know that $\min(b_1)\leq\min(b_2)$, and so $\min(F_1)\leq\min(F_2)$.  Consider Figure~\ref{fig:domsSE}. The first two images show the case where $F_1$ is type 1 and the next  two show the analogous situation in the case where $F_1$ is type 2. Since $T$ is a set-valued domino tableau and $\min(F_1)\leq\min(F_2)$, $F_2$ cannot intersect region $B$. Also, the image shows that $F_2$ cannot lie completely in region $C$ since it must be the same type as $F_1$. Thus $F_2$ must intersect region $A$ and so is weakly southeast of $F_1$. Since we know $F_2$ lies on the diagonal to the right of the diagonal of $F_1$, this implies $\max(F_1)\leq\min(F_2)$. Hence $\max(b_1)\leq\min(b_2)$, as desired.

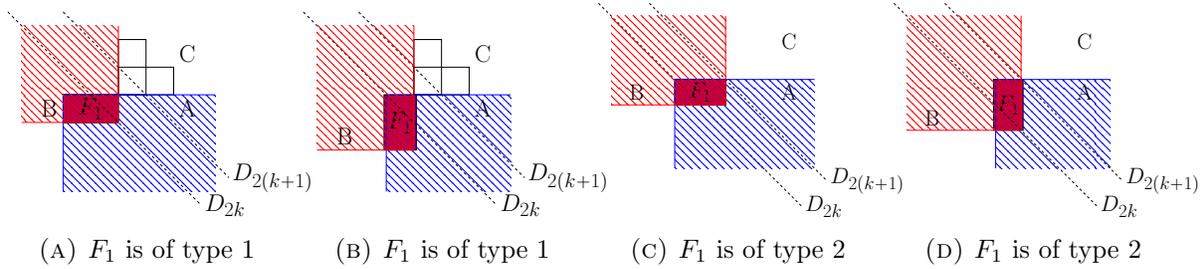
\begin{figure}[h]
\begin{center}
\begin{subfigure}[b]{0.22\linewidth}
\centering
\resizebox{4.2cm}{!}{
\begin{tikzpicture}[node distance=0cm, outer sep=0pt]
     \tikzstyle{every node}=[font=\Huge]
     
     \draw ( 5, 5) node {C};
     
     \node[bhor, fill = purple] (1) at ( 1.5, 3) {$F_1$};
     \node[bver] (4) at ( 3, 4.5) {};
     \node[bhor] (6) at ( 3.5, 4) {};
     
     \draw [thick, dashed] ( -1.5, 6.5) -- ( 5.5, -0.5);
     \draw ( 5.5, -0.5) node[right] {$D_{2k}$};
     \draw [thick, dashed] ( 0.5, 6.5) -- ( 6.5, 0.5);
     \draw ( 6.5, 0.5) node[right] {$D_{2(k+1)}$};
     
     \draw[color=blue] (0.5, 3.5) -- (6, 3.5);
     \draw[color=blue] (0.5, 3.5) -- (0.5, 0);
     \fill[draw=none, pattern=north west lines, pattern color=blue] (0.5,3.5) rectangle (6,0);
     \draw[thick] ( 5, 3) node {A};
     
     \draw[color=red] (-1, 2.5) -- (2.5, 2.5);
     \draw[color=red] (2.5, 6) -- (2.5, 2.5);
     \fill[draw=none, pattern=north west lines, pattern color=red] (-1,6) rectangle (2.5, 2.5);
     \draw[thick] ( 0, 3) node {B};
     
\end{tikzpicture}
}
\caption{$F_1$ is of type 1}
\end{subfigure}
\begin{subfigure}[b]{0.22\linewidth}
\centering
\resizebox{4.2cm}{!}{
\begin{tikzpicture}[node distance=0cm, outer sep=0pt]

     \tikzstyle{every node}=[font=\Huge]
     
     \draw ( 5, 5) node {C};
     
     \node[bver, fill = purple] (1) at ( 2, 2.5) {$F_1$};
     \node[bver] (4) at ( 3, 4.5) {};
     \node[bhor] (6) at ( 3.5, 4) {};
     
     \draw [thick, dashed] ( -1.5, 6.5) -- ( 5.5, -0.5);
     \draw ( 5.5, -0.5) node[right] {$D_{2k}$};
     \draw [thick, dashed] ( 0.5, 6.5) -- ( 6.5, 0.5);
     \draw ( 6.5, 0.5) node[right] {$D_{2(k+1)}$};
     
     \draw[color=blue] (1.5, 3.5) -- (6, 3.5);
     \draw[color=blue] (1.5, 3.5) -- (1.5, 0);
     \fill[draw=none, pattern=north west lines, pattern color=blue] (1.5,3.5) rectangle (6,0);
     \draw[thick] ( 5, 3) node {A};
     
     \draw[color=red] (-1, 1.5) -- (2.5, 1.5);
     \draw[color=red] (2.5, 6) -- (2.5, 1.5);
     \fill[draw=none, pattern=north west lines, pattern color=red] (-1,6) rectangle (2.5, 1.5);
     \draw[thick] ( 0, 2) node {B};
     
\end{tikzpicture}
}
\caption{$F_1$ is of type 1}
\end{subfigure}
\begin{subfigure}[b]{0.22\linewidth}
\centering
\resizebox{4.2cm}{!}{
\begin{tikzpicture}[node distance=0cm, outer sep=0pt]

     \tikzstyle{every node}=[font=\Huge]
     
     \draw ( 5, 6) node {C};
     
     \node[bhor, fill = purple] (1) at ( 1.5, 4) {$F_1$};
     
     \draw [thick, dashed] ( -2.5, 7.5) -- ( 5.5, -0.5);
     \draw ( 5.5, -0.5) node[right] {$D_{2k}$};
     \draw [thick, dashed] ( -0.5, 7.5) -- ( 6.5, 0.5);
     \draw ( 6.5, 0.5) node[right] {$D_{2(k+1)}$};
     
     \draw[color=blue] (0.5, 4.5) -- (6, 4.5);
     \draw[color=blue] (0.5, 4.5) -- (0.5, 1);
     \fill[draw=none, pattern=north west lines, pattern color=blue] (0.5,4.5) rectangle (6,1);
     \draw[thick] ( 5, 4) node {A};
     
     \draw[color=red] (-2, 3.5) -- (2.5, 3.5);
     \draw[color=red] (2.5, 7) -- (2.5, 3.5);
     \fill[draw=none, pattern=north west lines, pattern color=red] (-2,7) rectangle (2.5, 3.5);
     \draw[thick] ( -1, 4) node {B};
     
\end{tikzpicture}
}
\caption{$F_1$ is of type 2}
\end{subfigure}
\begin{subfigure}[b]{0.22\linewidth}
\centering
\resizebox{4.2cm}{!}{
\begin{tikzpicture}[node distance=0cm, outer sep=0pt]

     \tikzstyle{every node}=[font=\Huge]
     
     \draw ( 5, 6) node {C};
     
     \node[bver, fill = purple] (1) at ( 2, 3.5) {$F_1$};
     
     \draw [thick, dashed] ( -2.5, 7.5) -- ( 5.5, -0.5);
     \draw ( 5.5, -0.5) node[right] {$D_{2k}$};
     \draw [thick, dashed] ( -0.5, 7.5) -- ( 6.5, 0.5);
     \draw ( 6.5, 0.5) node[right] {$D_{2(k+1)}$};
     
     \draw[color=blue] (1.5, 4.5) -- (6, 4.5);
     \draw[color=blue] (1.5, 4.5) -- (1.5, 1);
     \fill[draw=none, pattern=north west lines, pattern color=blue] (1.5,4.5) rectangle (6,1);
     \draw[thick] ( 5, 4) node {A};
     
     \draw[color=red] (-2, 2.5) -- (2.5, 2.5);
     \draw[color=red] (2.5, 7) -- (2.5, 2.5);
     \fill[draw=none, pattern=north west lines, pattern color=red] (-2,7) rectangle (2.5, 2.5);
     \draw[thick] ( -1, 3) node {B};
     
\end{tikzpicture}
}
\caption{$F_1$ is of type 2}
\end{subfigure}
\caption{$F_2$ is southeast of $F_1$ when $b_1$ is left of $b_2$ in $t_i$.}
\label{fig:domsSE}
\end{center}
\end{figure}

We next show that the entries of $t_1$ and $t_2$ strictly increase down columns. Suppose box $b_1$ lies directly above box $b_2$ in $t_i$ and let $F_1$ and $F_2$ be as before. Note that $F_2$ lies on the diagonal directly left of that of $F_1$. Since taking the smallest entry in each domino of $T$ gives a domino tableau and $\Gamma^*$ restricts to $\Gamma$ on domino tableaux, we know that $\min(b_1)<\min(b_2)$ and so $\min(F_1)<\min(F_2)$. Consider Figure~\ref{fig:domsstrictSE}. Since $T$ is a set-valued domino tableau and $\min(F_1)<\min(F_2)$, $F_2$ cannot intersect the region $B$. We can also see that $F_2$ cannot be completely contained in region $C$, and hence $F_2$ intersects region $A$ and is weakly southeast of $F_1.$ Since $F_2$ lies on the diagonal to the left of that of $F_1$, then $\max(F_1)<\min(F_2)$. We conclude  that $\max(b_1)<\min(b_2)$, as desired.

\begin{figure}[h]
\begin{center}
\begin{subfigure}[b]{0.22\linewidth}
\centering
\resizebox{4.2cm}{!}{
\begin{tikzpicture}[node distance=0cm, outer sep=0pt]

     \tikzstyle{every node}=[font=\Huge]
     
     \draw ( -1, 1) node {C};
     
     \node[bhor, fill = purple] (1) at ( 1.5, 3) {$F_1$};
     
     \draw [thick, dashed] ( -1.5, 6.5) -- ( 5.5, -0.5);
     \draw ( 5.5, -0.5) node[right] {$D_{2k}$};
     \draw [thick, dashed] ( -3.5, 6.5) -- ( 3.5, -0.5);
     \draw ( 3.5, -1.5) node[right] {$D_{2(k-1)}$};
     
     \draw[color=blue] (0.5, 3.5) -- (5, 3.5);
     \draw[color=blue] (0.5, 3.5) -- (0.5, 0);
     \fill[draw=none, pattern= north west lines, pattern color=blue] (0.5,3.5) rectangle (5,0);
     \draw[thick] ( 4, 3) node {A};
     
     \draw[color=red] (-2, 2.5) -- (2.5, 2.5);
     \draw[color=red] (2.5, 6) -- (2.5, 2.5);
     \fill[draw=none, pattern=north west lines, pattern color=red] (-2,6) rectangle (2.5, 2.5);
     \draw[thick] ( 2, 5) node {B};
     
\end{tikzpicture}
}
\caption{$F_1$ is of type 1}
\end{subfigure}
\begin{subfigure}[b]{0.22\linewidth}
\centering
\resizebox{4.2cm}{!}{
\begin{tikzpicture}[node distance=0cm, outer sep=0pt]

     \tikzstyle{every node}=[font=\Huge]
     
     \draw ( -1, 0.5) node {C};
     
     \node[bver, fill = purple] (1) at ( 2, 2.5) {$F_1$};
     
     \draw [thick, dashed] ( -1.5, 6.5) -- ( 5.5, -0.5);
     \draw ( 5.5, -0.5) node[right] {$D_{2k}$};
     \draw [thick, dashed] ( -3.5, 6.5) -- ( 3.5, -0.5);
     \draw ( 3.5, -1.5) node[right] {$D_{2(k-1)}$};
     
     \draw[color=blue] (1.5, 3.5) -- (5, 3.5);
     \draw[color=blue] (1.5, 3.5) -- (1.5, 0);
     \fill[draw=none, pattern=north west lines, pattern color=blue] (1.5,3.5) rectangle (5,0);
     \draw[thick] ( 4, 3) node {A};
     
     \draw[color=red] (-2, 1.5) -- (2.5, 1.5);
     \draw[color=red] (2.5, 6) -- (2.5, 1.5);
     \fill[draw=none, pattern=north west lines, pattern color=red] (-2,6) rectangle (2.5, 1.5);
     \draw[thick] ( 2, 5) node {B};
     
\end{tikzpicture}
}
\caption{$F_1$ is of type 1}
\end{subfigure}
\begin{subfigure}[b]{0.22\linewidth}
\centering
\resizebox{4.2cm}{!}{
\begin{tikzpicture}[node distance=0cm, outer sep=0pt]

     \tikzstyle{every node}=[font=\Huge]
     
     \draw ( -2, 2) node {C};
     
     \node[bhor, fill = purple] (1) at ( 1.5, 4) {$F_1$};
     \node[bver] (4) at ( 0, 2.5) {};
     \node[bhor] (6) at ( -0.5, 3) {};
     
     \draw [thick, dashed] ( -1.5, 6.5) -- ( 5.5, -0.5);
     \draw ( 5.5, -0.5) node[right] {$D_{2k}$};
     \draw [thick, dashed] ( -3.5, 6.5) -- ( 3.5, -0.5);
     \draw ( 3.5, -1.5) node[right] {$D_{2(k-1)}$};
     
     \draw[color=blue] (0.5, 4.5) -- (5, 4.5);
     \draw[color=blue] (0.5, 4.5) -- (0.5, 1);
     \fill[draw=none, pattern=north west lines, pattern color=blue] (0.5,4.5) rectangle (5,1);
     \draw[thick] ( 4, 4) node {A};
     
     \draw[color=red] (-3, 3.5) -- (2.5, 3.5);
     \draw[color=red] (2.5, 6) -- (2.5, 3.5);
     \fill[draw=none, pattern=north west lines, pattern color=red] (-3,6) rectangle (2.5, 3.5);
     \draw[thick] ( 2, 5) node {B};
     
\end{tikzpicture}
}
\caption{$F_1$ is of type 2}
\end{subfigure}
\begin{subfigure}[b]{0.22\linewidth}
\centering
\resizebox{4.2cm}{!}{
\begin{tikzpicture}[node distance=0cm, outer sep=0pt]

     \tikzstyle{every node}=[font=\Huge]
     
     \draw ( -1, 1) node {C};
     
     \node[bver, fill = purple] (1) at ( 2, 3.5) {$F_1$};
     \node[bver] (4) at ( 1, 1.5) {};
     \node[bhor] (6) at ( 0.5, 2) {};
     
     \draw [thick, dashed] ( -1.5, 6.5) -- ( 5.5, -0.5);
     \draw ( 5.5, -0.5) node[right] {$D_{2k}$};
     \draw [thick, dashed] ( -3.5, 6.5) -- ( 3.5, -0.5);
     \draw ( 3.5, -1.5) node[right] {$D_{2(k-1)}$};
     
     \draw[color=blue] (1.5, 4.5) -- (5, 4.5);
     \draw[color=blue] (1.5, 4.5) -- (1.5, 0);
     \fill[draw=none, pattern=north west lines, pattern color=blue] (1.5,4.5) rectangle (5,0);
     \draw[thick] ( 4, 4) node {A};
     
     \draw[color=red] (-2, 2.5) -- (2.5, 2.5);
     \draw[color=red] (2.5, 6) -- (2.5, 2.5);
     \fill[draw=none, pattern=north west lines, pattern color=red] (-2,6) rectangle (2.5, 2.5);
     \draw[thick] ( 2, 5) node {B};
     
\end{tikzpicture}
}
\caption{$F_1$ is of type 2}
\end{subfigure}
\caption{$F_2$ is southeast of $F_1$ when $b_1$ is over $b_2$ in $t_i$.}
\label{fig:domsstrictSE}
\end{center}
\end{figure}
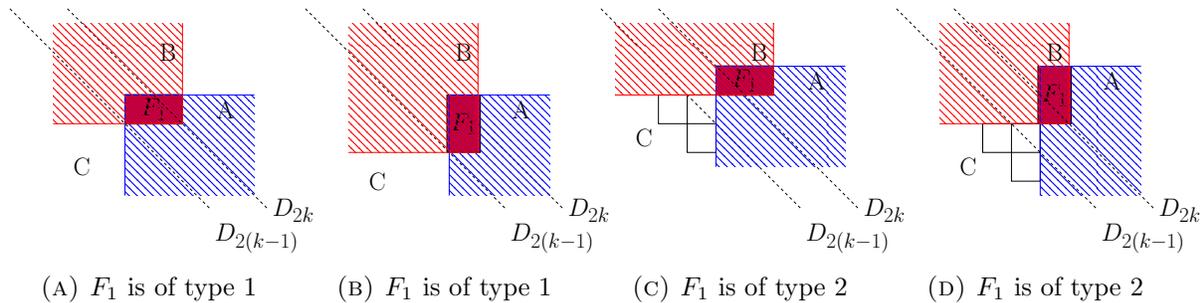

Note that we know that $(t_1,t_2)$ has shape $(\mu,\nu)$ by restricting $\Gamma^*$ to the set of domino tableaux of shape $\lambda$. We conclude that $\Gamma^*$ sends a set-valued domino tableau of shape $\lambda$ to a pair of semistandard set-valued tableaux of shape $(\mu, \nu)$, where $(\mu, \nu)$ is the 2-quotient of $\lambda$.

We will now describe $\Gamma^{*-1}$, the inverse map of $\Gamma^{*}$. Let $(t_1',t_2')$ be the semistandard tableaux obtained from $(t_1,t_2)$ by taking only the smallest entry in each box. We may then apply $\Gamma^{-1}$ to $(t_1',t_2')$ to obtain a domino tableau $T'$. We then define $\Gamma^{*-1}(t_1,t_2)$ to be the set-valued domino tableau obtained from $T'$ by reuniting each entry in $T'$ with the rest of the subset that was with that entry in $(t_1,t_2)$. We describe this precisely below. 

Let $(t_1, t_2)$ be a pair of semistandard set-valued tableaux of shape $(\mu, \nu)$. Similarly to the description of $\Gamma ^{-1}$, we recursively construct the set-valued domino tableau of shape $\lambda$ associated to $(t_1, t_2)$. At any step, we have a pair of set-valued tableaux $(t_1^{(i)}, t_2^{(i)})$ of shape $(\mu^{(i)}, \nu^{(i)})$ and the associated set-valued domino tableau $T^{(i)}$ of shape $\lambda^{(i)}$. We start the algorithm with $\mu^{(0)} = \nu^{(0)} = \lambda^{(0)} = \varnothing$. The algorithm stops when $(t_1^{(s)}, t_2^{(s)}) = (t_1, t_2)$. Then we have that the set-valued domino tableau associated to $(t_1, t_2)$ is $T^{(s)}$. 

We now describe the $i ^{th}$ step of the algorithm. Let $u_i$ be the smallest value appearing as the minimum entry in a box in $(t_1, t_2)$ that does not appear as the minimum entry in a box in $(t_1^{(i-1)}, t_2^{(i-1)})$. 
We build $(t_1^{(i)}, t_2^{(i)})$ of shape $(\mu ^{(i)}, \nu ^{(i)})$ by adding to $(t_1^{(i-1)}, t_2^{(i-1)})$ all cells of $(t_1, t_2)$ filled by a set with minimum $u_i$, while preserving their original position. 

To construct the domino tableau $T^{(i)}$ of shape $\lambda^{(i)}$, we use the following procedure for each diagonal, starting with the bottom diagonal:
For any cell in $t_1$ (resp. $t_2$) filled with a set with minimum $u_i$ on diagonal $D_k$, we add to $T^{(i-1)}$ a type 1 domino (resp. a type 2 domino) filled with that set on the corresponding diagonal $D_{2k}$. We then get the associated domino tableau $T^{(i)}$ of shape $\lambda^{(i)}$.

Below is an example of $\Gamma^{*-1}$ applied to the pair of set-valued tableaux \[(t_1, t_2) = \left( \raisebox{.1in}{ \ \begin{ytableau} \scriptstyle{3,4} & \scriptstyle{4,7} \\ \scriptstyle{5} \end{ytableau}, \ \begin{ytableau} \scriptstyle{1,2} & \scriptstyle{4} & \scriptstyle{4,5,6} \\ \scriptstyle{3,6 }\end{ytableau} \ } \right).\] Notice that we recover the tableau $T$ from the previous example.

\begin{minipage}{.53\textwidth}
\raisebox{.07in}{$(1)$ $\left( \ \varnothing, \ \begin{ytableau}\scriptstyle{1,2} \end{ytableau} \ \right) \rightarrow $ }
\resizebox{1.2 cm}{!}{
\begin{tikzpicture}[node distance=0 cm,outer sep = 0pt]
        \tikzstyle{every node}=[font=\Large]
        \node[bhor] (1) at ( 1.5, 3) {1,2};
\end{tikzpicture}
}\end{minipage} 
\begin{minipage}{.53\textwidth}
\raisebox{.28in}{$(2)$ $ \left( \ \begin{ytableau} \scriptstyle{3,4} \end{ytableau} , \ \begin{ytableau} \scriptstyle{1,2} \\ \scriptstyle{3,6} \end{ytableau} \ \right) \rightarrow $ }
\resizebox{1.2cm}{!}{
\begin{tikzpicture}[node distance=0 cm,outer sep = 0pt]
        \tikzstyle{every node}=[font=\Large]
        \node[bhor] (1) at ( 1.5, 1) {1,2};
        \node[bver] (2) at ( 1, -.5) {3,6};
        \node[bver] (3) at ( 2, -.5) {3,4};
\end{tikzpicture}
}\end{minipage}
\begin{center}

\begin{minipage}{.53\textwidth}
\raisebox{.28in}{$(3)$ $ \left( \ \begin{ytableau} \scriptstyle{3,4} & \scriptstyle{4,7} \end{ytableau}, \ \begin{ytableau} \scriptstyle{1,2} & \scriptstyle{4} & \scriptstyle{4,5,6} \\ \scriptstyle{3,6} \end{ytableau} \ \right) \rightarrow $ }
\resizebox{3.3cm}{!}{
\begin{tikzpicture}[node distance=0 cm,outer sep = 0pt]
        \tikzstyle{every node}=[font=\Large]
        \node[bhor] (1) at ( 1.5, 1) {1,2};
        \node[bver] (2) at ( 1, -.5) {3,6};
        \node[bver] (3) at ( 2, -.5) {3,4};
        \node[bver] (4) at ( 3, .5)  {4,7};
        \node[bver] (6) at ( 4, .5)  {4};
        \node[bhor] (7) at ( 5.5, 1) {4,5,6};
\end{tikzpicture}}
\end{minipage}

\begin{minipage}{.53\textwidth}
\raisebox{.4in}{$(4)$ $ \left( \ \begin{ytableau} \scriptstyle{3,4} & \scriptstyle{4,7} \\ \scriptstyle{5} \end{ytableau}, \ \begin{ytableau} \scriptstyle{1,2} & \scriptstyle{4} & \scriptstyle{4,5,6} \\ \scriptstyle{3,6} \end{ytableau} \ \right) \rightarrow $ } 
\resizebox{3.3cm}{!}{
\begin{tikzpicture}[node distance=0 cm,outer sep = 0pt]
        \tikzstyle{every node}=[font=\Large]
        \node[bhor] (1) at ( 1.5, 3) {1,2};
        \node[bver] (2) at ( 1, 1.5) {3,6};
        \node[bver] (3) at ( 2, 1.5) {3,4};
        \node[bver] (4) at ( 3, 2.5) {4,7};
        \node[bhor] (5) at ( 1.5, 0) {5};
        \node[bver] (6) at ( 4, 2.5) {4};
        \node[bhor] (7) at ( 5.5, 3) {4,5,6};
\end{tikzpicture}
}\end{minipage}
\end{center}

We need to see that $\Gamma^{*-1}$ gives a set-valued domino tableau. Since $\Gamma^{-1}$ and $\Gamma^{*-1}$ coincide on pairs of semistandard tableaux and give a domino tableau, we see that the minimum entries of the dominoes increase weakly along rows and strictly down columns. Also, $T = \Gamma^{*-1} (t_1, t_2)$ has shape $\lambda$, since $\Gamma^{-1} (t_1', t_2')$ does.

Suppose $F_1$ and $F_2$ are dominoes of  the same type of $T = \Gamma^{*-1} (t_1, t_2)$, that $F_1$ is on diagonal $D_{2k}$ for some $k$, and that $F_2$ is on diagonal $D_{2(k+1)}$ and is weakly southeast of $F_1$. We must show that $\max(F_1)\leq\min(F_2)$.

Let $b_1$ and $b_2$ be the boxes of $t_i$ such that $\Gamma^{*-1}$ sends the entries of $b_1$ to $F_1$ and the entries of $b_2$ to $F_2$. Then $b_1$ lies on diagonal $D_{k}$ of $t_i$ and $b_2$ lies on diagonal $D_{k+1}$. Then $b_2$ is either weakly southeast of $b_1$ or is weakly northwest of $b_1$ as shown in the first image of Figure~\ref{fig:cellspositions}. However, since $\Gamma^{*-1}$ restricts to $\Gamma^{-1}$ on semistandard tableaux and gives a domino tableau, we know that $\min(F_1)\leq\min(F_2)$. It follows that $\min(b_1)\leq\min(b_2)$, and so $b_2$ must be weakly southeast of $b_1$. Thus $\max(b_1)\leq\min(b_2)$, which implies that $\max(F_1)\leq\min(F_2)$.

Lastly, suppose that $F_1$ and $F_2$ are dominoes of the same type of $T = \Gamma^{*-1} (t_1, t_2)$, that $F_1$ is on diagonal $D_{2(k+1)}$ for some $k$, and that $F_2$ is on diagonal $D_{2k}$ and is weakly southeast of $F_1$. We must show that $\max(F_1)<\min(F_2)$.

Let $b_1$ and $b_2$ be as before, so $b_1$ lies on diagonal $D_{k+1}$ of $t_i$ and $b_2$ lies on diagonal $D_{k}$. From the second image of Figure~\ref{fig:cellspositions}, we see that either $b_2$ is weakly southeast of $b_1$ or $b_2$ is weakly northwest of $b_1$. For the same reason as in the previous argument, we know that $\min(F_1)<\min(F_2)$, so $\min(b_1)<\min(b_2)$. This means that $b_2$ must lie weakly southeast of $b_1$, and so $\max(b_1)<\min(b_2)$. We then have that $\max(F_1)<\min(F_2)$, as desired.

\begin{figure}[h]
\resizebox{8cm}{!}{
\begin{tikzpicture}[node distance=0 cm,outer sep = 0pt]
        
        \tikzstyle{every node}=[font=\Huge]
        \node[bsq] (1) at ( 1.5, 3) {$b_1$};
        \node[bsq, fill=blue] (2) at ( 2.5, 3) {$b_2$};
        \node[bsq, fill=blue] (3) at ( 3.5, 2) {$b_2$};
        \node[bsq, fill=blue] (4) at ( 4.5, 1) {$b_2$};
        \node[bsq, fill=red] (5) at ( 1.5, 4) {$b_2$};
        \node[bsq, fill=red] (6) at ( 0.5, 5) {$b_2$};
        \node[bsq] (6) at ( 0.5, 4) {};
        \node[bsq] (7) at ( 2.5, 2) {};
        \node[bsq] (8) at ( 3.5, 1) {};
        \node[bsq] (9) at ( -0.5, 5) {};
        
        \draw [dashed, thick] ( -1.5, 6) -- ( 5.5, -1);
        \draw ( 5.5, -2) node[right] {$D_{k}$};
        \draw [dashed, thick] ( -0.5, 6) -- ( 6.5, -1);
        \draw ( 6.5, -1) node[right] {$D_{k+1}$};
        

\end{tikzpicture}
\begin{tikzpicture}[node distance=0 cm,outer sep = 0pt]
        
       \tikzstyle{every node}=[font=\Huge] 
       \node[bsq, fill=red] (1) at ( 1.5, 3) {$b_2$};
        \node[bsq] (2) at ( 2.5, 3) {$b_1$};
        \node[bsq] (3) at ( 3.5, 2) {};
        \node[bsq] (4) at ( 4.5, 1) {};
        \node[bsq] (5) at ( 1.5, 4) {};
        \node[bsq] (6) at ( 0.5, 5) {};
        \node[bsq, fill=red] (6) at ( 0.5, 4) {$b_2$};
        \node[bsq, fill=blue] (7) at ( 2.5, 2) {$b_2$};
        \node[bsq, fill=blue] (8) at ( 3.5, 1) {$b_2$};
        \node[bsq, fill=red] (9) at ( -0.5, 5) {$b_2$};
        
        \draw [dashed, thick] ( -1.5, 6) -- ( 5.5, -1);
        \draw ( 5.5, -2) node[right] {$D_{k}$};
        \draw [dashed, thick] ( -0.5, 6) -- ( 6.5, -1);
        \draw ( 6.5, -1) node[right] {$D_{k+1}$};
        

\end{tikzpicture}
}
\caption{Possible relative positions of $b_1$ and $b_2$: case 1 is shown in red, where $\max(b_2)\leq \min(b_1)$ and case 2 in blue, where $\max(b_1)\leq \min(b_2)$.}
\label{fig:cellspositions}
\end{figure}
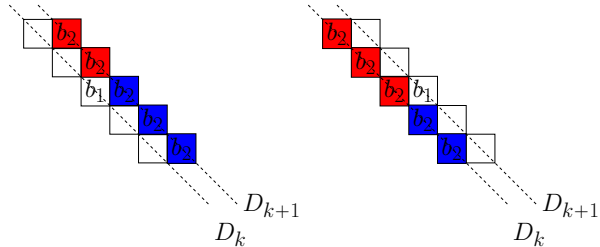 

We conclude that $\Gamma^{*-1}$ sends a pair of semistandard set-valued tableaux of shape $(\mu, \nu)$ to a set-valued domino tableau of shape $\lambda$. It is clear that $\Gamma^{*-1}$ and $\Gamma^*$ are indeed inverses as they are governed by $\Gamma^{-1}$ and $\Gamma$.

\end{proof}

The translation into the language of symmetric functions gives us the following theorem. Note that for pavable partition $\lambda$, $\frac{1}{2}|\lambda|$ gives the number of dominoes in a paving of $\lambda$.

\begin{corollary}
Let $\lambda$ be a partition with 2-quotient ($\mu$, $\nu$). Then 
\[
G_{\mu} G_{\nu} = \sum _{T} (-1)^{|T|-{\frac{1}{2}|\lambda|}} x^T,\]
where we sum over all set-valued domino tableaux of shape $\lambda$.
\end{corollary}

\begin{proof}
Consider a term in the product $G_\mu G_\nu$. This monomial corresponds to a pair of semistandard set-valued tableaux: $t_1$ of shape $\mu$ and $t_2$ of shape $\nu$. The pair $(t_1,t_2)$ corresponds to some set-valued domino tableau $T$ of shape $\lambda$ by Theorem~\ref{thm:bijectiondomino2}. It is clear from the previous bijection that $x^{t_1} x^{t_2} = x^T$.

We now examine the sign of $x^{t_1}x^{t_2}$ in $G_\mu G_\nu$. We see that it appears with sign 
\[(-1)^{|t_1|-|\mu|}(-1)^{|t_2|-|\nu|}=(-1)^{|t_1|+|t_2|-(|\mu|+|\nu|)}=(-1)^{|T|-\frac{1}{2}|\lambda|}.\] This gives the desired result.
\end{proof}

\section{$Q$-Schur functions and shifted domino tableaux}\label{sec:shifted}
\subsection{$Q$-Schur functions}

Let $\lambda$ be a partition. Define $up(\lambda)$ to be the boxes of $\lambda$ that lie on a diagonal weakly northeast of $D_0$ and $down(\lambda)$ to be the boxes of $\lambda$ that lie on a diagonal strictly southwest of $D_0$.

Let $\lambda = (\lambda_1, \lambda_2, \ldots \lambda_k)$ be such that $\lambda_k \geqslant k$. We may form a \textit{shifted Young tableau} of shape $\lambda$ by filling the boxes of $down(\lambda)$ with the symbol $X$ and filling the boxes of $up(\lambda)$ with primed and unprimed positive integers with linear order $1'<1<2'<2<\ldots$ such that

\begin{itemize}
\item rows and columns are weakly increasing,
\item there is at most one occurrence of $i'$ in any row and
\item there is at most one occurrence of $i$ in any column.
\end{itemize}
For example, both tableaux below are shifted Young tableaux.

\begin{center}
\ytableausetup{boxsize=.2in}
\begin{ytableau}
1' & 1 & 2' \\
X & 2 & 4
\end{ytableau}\hspace{1in}
\begin{ytableau}
2 & 3' & 3 & 3 \\
X & 3' & 4 \\
X & X & 6
\end{ytableau}
\end{center}

We may associate a monomial to a shifted Young tableau $T$ by defining 
\[x^T=x_1^{\beta_1(T)}x_2^{\beta_2(T)}\cdots,\]
where $\beta_i(T)$ is the number of occurrences of $i$ and $i'$ in $T$. The tableau above on the right corresponds to monomial $x_2x_3^4x_4x_6$.

The \textit{$Q$-Schur function} indexed by partition $\lambda=(\lambda_1,\ldots,\lambda_k)$ with $\lambda_k\geq k$, denoted $Q_\lambda$, is then defined to be the weighted generating function over all shifted Young tableaux of shape $\lambda$:
\[Q_\lambda= \sum_{sh(T)=\lambda} x^T.\]

The $Q$-Schur functions were introduced by I. Schur in relation to the projective representations of the symmetric and alternating groups \cite{schur1911darstellung}. They have since been widely studied, for example by B. Sagan \cite{sagan1987shifted} and J. Stembridge \cite{stembridge1989shifted}.

Below are a few terms of $Q_{(3,3,3)}$ and the corresponding shifted tableaux. Note that each term shown appears with multiplicity in the full $Q_{(3,3,3)}$. For example, $x_1^3x_2^2x_3$ appears with coefficient 8 since each element on the diagonal of the tableau shown on the left may be primed or unprimed. 

\[Q_{(3,3,3)} = x_1^{3} x_2^2 x_3 + x_1^{3} x_2^{2} x_4 + x_1^{3} x_2^{2} x_5 + x_1 x_2^{2} x_3^2 x_4 + x_1 x_2^{3} x_3^2 + \ldots\]

\begin{center}
\ytableausetup{boxsize=.2in}
\begin{ytableau}
1' & 1 & 1 \\
X & 2' & 2 \\
X & X & 3'
\end{ytableau}\hspace{.2in}
\begin{ytableau}
1' & 1 & 1 \\
X & 2 & 2 \\
X & X & 4
\end{ytableau}\hspace{.2in}
\begin{ytableau}
1' & 1 & 1 \\
X & 2' & 2 \\
X & X & 5'
\end{ytableau}\hspace{.2in}
\begin{ytableau}
1' & 2' & 3' \\
X & 2' & 3' \\
X & X & 4
\end{ytableau}\hspace{.2in}
\begin{ytableau}
1' & 2' & 2 \\
X & 2' & 3' \\
X & X & 3
\end{ytableau}
\end{center}

\subsection{Shifted domino tableaux}
We next define the notion of a shifted domino tableau, which was first introduced by Z. Chemli \cite{chemli}. 

\begin{definition}\cite{chemli}\label{def:shiftedpaving} Let $\lambda$ be a pavable partition with 2-quotient $(\mu=(\mu_1,\ldots,\mu_s),\nu=(\nu_1,\ldots,\nu_t))$ and fixed paving. This paving is a \textit{shifted paving} if 
\begin{itemize}
\item $\mu_s\geq s$ and $\nu_t\geq t$ and 
\item there is no vertical domino $d$ on $D_0$ such that the dominoes directly left of $d$ and adjacent to $d$ are all strictly below $D_0$. 
\end{itemize}
If such a paving of $\lambda$ exists, we call $\lambda$ a \textit{shifted pavable partition}.
\end{definition}
For a shifted pavable partition $\lambda$ with fixed shifted paving, define $up(\lambda)$ to be the dominoes of $\lambda$ that lie on a diagonal weakly northeast of $D_0$ and $down(\lambda)$ to be the dominoes of $\lambda$ that lie on a diagonal strictly southwest of $D_0$. 

\begin{definition}\cite{chemli}
Given a shifted pavable partition $\lambda$ with fixed shifted paving, a \textit{shifted domino tableau} is a filling of the dominoes of $down(\lambda)$ with $X$ and the dominoes of $up(\lambda)$ with primed and unprimed integers with linear order $1'<1<2'<2<\ldots$ such that 
\begin{itemize}
\item rows and columns are weakly increasing,
\item there is at most one occurrence of $i$ in any column, and 
\item there is at most one occurrence of $i'$ in any row.
\end{itemize}

For $T$ a shifted domino tableau, let $up(T)$ be the dominoes of $T$ that lie on a diagonal weakly northeast of $D_0$ along with the filling of these dominoes. We consider two shifted domino tableaux $T$ and $T'$ of shape $\lambda$ to be \textit{equivalent} if $up(T)=up(T')$. Let the  \textit{set of shifted domino tableaux} refer to the set up to equivalence. 
\end{definition}

We define the \textit{diagonal reading word} of a shifted domino tableau to be the sequence obtained from reading northwest to southeast along each diagonal $D_{2k}$ with $k\geq0$, starting with $D_0$ and separating entries on distinct diagonals with ``$/$''. Note that equivalent domino tableaux have equal diagonal reading words. 

\begin{example}\label{ex:shifteddominotab}
Let $\lambda = (6,5,5,4)$, which is a pavable partition with 2-quotient $((2,2), (3,3))$. Clearly, this 2-quotient respects the previous conditions of Definition~\ref{def:shiftedpaving}. Below are two different pavings of $\lambda$. The paving on the left is a shifted paving, and the one on the right is not. The problematic domino is highlighted in the second paving.
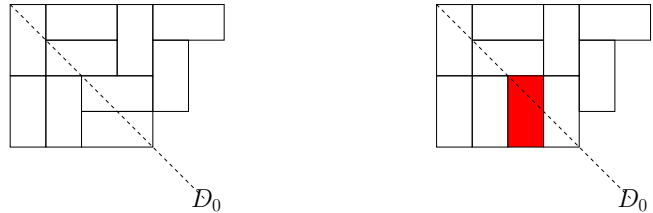
\begin{center}
\resizebox{3cm}{!}{
\begin{tikzpicture}[node distance=0 cm,outer sep = 0pt]
        \tikzstyle{every node}=[font=\huge]
        \node[bver] (1) at ( 1, 3) {};
        \node[bhor] (2) at ( 2.5, 3.5) {};
        \node[bhor] (3) at ( 2.5, 2.5) {};
        \node[bver] (4) at ( 4, 3) {};
        \node[bhor] (5) at ( 5.5, 3.5) {};
        \node[bver] (6) at ( 1, 1) {};
        \node[bver] (7) at ( 2, 1) {};
        \node[bhor] (8) at ( 3.5, 1.5) {};
        \node[bhor] (9) at ( 3.5, 0.5) {};
        \node[bver] (10) at ( 5, 2) {};
        
        \draw[dashed, thick] ( 0.5, 4) -- ( 6, -1.5) ;
        \draw (6, -1.5) node {$D_0$};
\end{tikzpicture}
}\hspace{1in}
\resizebox{3cm}{!}{
\begin{tikzpicture}[node distance=0 cm,outer sep = 0pt]
        \tikzstyle{every node}=[font=\huge]
        \node[bver] (1) at ( 1, 3) {};
        \node[bhor] (2) at ( 2.5, 3.5) {};
        \node[bhor] (3) at ( 2.5, 2.5) {};
        \node[bver] (4) at ( 4, 3) {};
        \node[bhor] (5) at ( 5.5, 3.5) {};
        \node[bver] (6) at ( 1, 1) {};
        \node[bver] (7) at ( 2, 1) {};
        \node[bver, fill = red] (8) at ( 3, 1) {};
        \node[bver] (9) at ( 4, 1) {};
        \node[bver] (10) at ( 5, 2) {};
        
        \draw[dashed, thick] ( 0.5, 4) -- ( 6, -1.5) ;
        \draw (6, -1.5) node {$D_0$};
\end{tikzpicture}
}
\end{center}

Let $\lambda = (5, 5, 4, 3, 3, 2)$, which is a pavable partition with 2-quotient $((3,1,1),(2,2,2))$. Since neither of the partitions of the 2-quotient respects condition 1 of Definition~\ref{def:shiftedpaving}, then $\lambda$ is not a shifted pavable partition.

The following tableaux $T$, $T'$, and $T''$ are equivalent shifted domino tableaux and are thus considered equal in the set of shifted domino tableaux. The diagonal reading word for each is 
\begin{center}1$',$ 1, 2$',$ 3$',$ 3 / 1, 2$',$ 3$'$ / 3$',$ 4 / 3.\end{center}
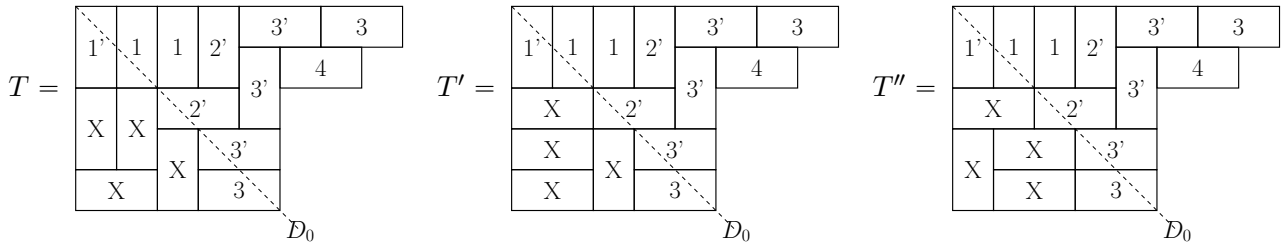
\begin{center}
\raisebox{2cm}{$T = $}
\resizebox{4.5cm}{!}{
\begin{tikzpicture}[node distance=0 cm,outer sep = 0pt]
        \tikzstyle{every node}=[font=\LARGE]
        \node[bver] (1) at ( 1, 3) {1'};
        \node[bver] (2) at ( 2, 3) {1};
        \node[bver] (3) at ( 3, 3) {1};
        \node[bver] (4) at ( 4, 3) {2'};
        \node[bhor] (5) at ( 5.5, 3.5) {3'};
        \node[bhor] (6) at ( 7.5, 3.5) {3};
        \node[bhor] (7) at ( 6.5, 2.5) {4};
        \node[bver] (8) at ( 5, 2) {3'};
        \node[bver] (10) at ( 1, 1) {X};
        \node[bver] (11) at ( 2, 1) {X};
        \node[bhor] (12) at ( 3.5, 1.5) {2'};
        \node[bhor] (13) at ( 4.5, 0.5) {3'};
        \node[bver] (14) at ( 3, 0) {X};
        \node[bhor] (15) at ( 1.5, -0.5) {X};
        \node[bhor] (16) at ( 4.5, -0.5) {3};
        
        \draw[dashed] ( 0.5, 4) -- ( 6, -1.5) ;
        \draw (6, -1.5) node {$D_0$};
\end{tikzpicture}
}\hspace{.1in}
\raisebox{2cm}{$T' = $}
\resizebox{4.5cm}{!}{
\begin{tikzpicture}[node distance=0 cm,outer sep = 0pt]
        \tikzstyle{every node}=[font=\LARGE]
        \node[bver] (1) at ( 1, 3) {1'};
        \node[bver] (2) at ( 2, 3) {1};
        \node[bver] (3) at ( 3, 3) {1};
        \node[bver] (4) at ( 4, 3) {2'};
        \node[bhor] (5) at ( 5.5, 3.5) {3'};
        \node[bhor] (6) at ( 7.5, 3.5) {3};
        \node[bhor] (7) at ( 6.5, 2.5) {4};
        \node[bver] (8) at ( 5, 2) {3'};
        \node[bhor] (10) at ( 1.5, 1.5) {X};
        \node[bhor] (11) at ( 1.5, 0.5) {X};
        \node[bhor] (12) at ( 3.5, 1.5) {2'};
        \node[bhor] (13) at ( 4.5, 0.5) {3'};
        \node[bver] (14) at ( 3, 0) {X};
        \node[bhor] (15) at ( 1.5, -0.5) {X};
        \node[bhor] (16) at ( 4.5, -0.5) {3};        
        \draw (6, -1.5) node {};
        
        \draw[dashed] ( 0.5, 4) -- ( 6, -1.5) ;
        \draw (6, -1.5) node {$D_0$};
\end{tikzpicture}
}\hspace{.1in}
\raisebox{2cm}{$T'' = $}
\resizebox{4.5cm}{!}{
\begin{tikzpicture}[node distance=0 cm,outer sep = 0pt]
        \tikzstyle{every node}=[font=\LARGE]
        \node[bver] (1) at ( 1, 3) {1'};
        \node[bver] (2) at ( 2, 3) {1};
        \node[bver] (3) at ( 3, 3) {1};
        \node[bver] (4) at ( 4, 3) {2'};
        \node[bhor] (5) at ( 5.5, 3.5) {3'};
        \node[bhor] (6) at ( 7.5, 3.5) {3};
        \node[bhor] (7) at ( 6.5, 2.5) {4};
        \node[bver] (8) at ( 5, 2) {3'};
        \node[bhor] (10) at ( 1.5, 1.5) {X};
        \node[bver] (11) at ( 1, 0) {X};
        \node[bhor] (12) at ( 3.5, 1.5) {2'};
        \node[bhor] (13) at ( 4.5, 0.5) {3'};
        \node[bhor] (14) at ( 2.5, 0.5) {X};
        \node[bhor] (15) at ( 2.5, -0.5) {X};
        \node[bhor] (16) at ( 4.5, -0.5) {3};
        
        \draw[dashed] ( 0.5, 4) -- ( 6, -1.5) ;
        \draw (6, -1.5) node {$D_0$};
\end{tikzpicture}
}
\end{center}
\end{example}


\subsection{Bijection between shifted domino tableaux and shifted Young tableaux}\label{sec:shiftedbijection}

Chemli proves the following bijection. 

\begin{theorem}[Theorem 3.1 \cite{chemli}]
Let $\lambda$ be a shifted pavable partition with 2-quotient $(\mu,\nu)$. The set of shifted domino tableaux of shape $\lambda$ is in bijection with the set of pairs $(t_1,t_2)$ of shifted Young tableaux of shape $(\mu,\nu)$.
\label{thm:bijectiondomino4}
\end{theorem}

The bijection is a slight modification of the map $\Gamma$ from Theorem~\ref{thm:bijectiondomino}. We call this modified map $\Gamma_{s}$ and it goes from the set of shifted domino tableaux of shape $\lambda$ to the set of pairs of shifted Young tableaux of shape $(\mu,\nu)$.

The modification of $\Gamma$ to obtain $\Gamma_s$ is quite simple. We apply $\Gamma$ to a shifted domino tableau as if it were a domino tableau, without considering if the dominoes are filled with $X $'s or with integers. For example, let us apply $\Gamma_s$ to the tableau $T$ of Example~\ref{ex:shifteddominotab}.  

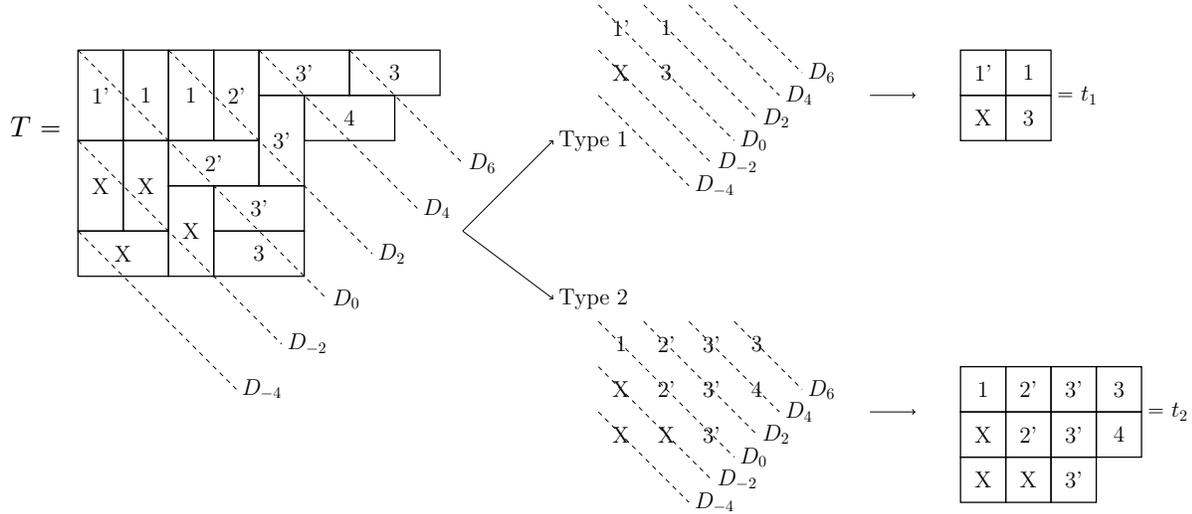
\begin{center}
\raisebox{2in}{$T = $}
\resizebox{15cm}{!}{
\begin{tikzpicture}[node distance=0 cm,outer sep = 0pt]

        \tikzstyle{every node}=[font=\Large]
        \node[bver] (1) at ( 1, 8) {1'};
        \node[bver] (2) at ( 2, 8) {1};
        \node[bver] (3) at ( 3, 8) {1};
        \node[bver] (4) at ( 4, 8) {2'};
        \node[bhor] (5) at ( 5.5, 8.5) {3'};
        \node[bhor] (6) at ( 7.5, 8.5) {3};
        \node[bhor] (7) at ( 6.5, 7.5) {4};
        \node[bver] (8) at ( 5, 7) {3'};
        \node[bver] (10) at ( 1, 6) {X};
        \node[bver] (11) at ( 2, 6) {X};
        \node[bhor] (12) at ( 3.5, 6.5) {2'};
        \node[bhor] (13) at ( 4.5, 5.5) {3'};
        \node[bver] (14) at ( 3, 5) {X};
        \node[bhor] (15) at ( 1.5, 4.5) {X};
        \node[bhor] (16) at ( 4.5, 4.5) {3};
        
        \draw[dashed] ( 0.5, 9) -- ( 6, 3.5) ;
        \draw (6, 3.5) node[right] {$D_0$};
        \draw[dashed] ( 0.5, 7) -- ( 5, 2.5) ;
        \draw (5, 2.5) node[right] {$D_{-2}$};
        \draw[dashed] ( 0.5, 5) -- ( 4, 1.5) ;
        \draw (4, 1.5) node[right] {$D_{-4}$};
        \draw[dashed] ( 2.5, 9) -- ( 7, 4.5) ;
        \draw (7, 4.5) node[right] {$D_{2}$};
        \draw[dashed] ( 4.5, 9) -- ( 8, 5.5) ;
        \draw (8, 5.5) node[right] {$D_{4}$};
        \draw[dashed] ( 6.5, 9) -- ( 9, 6.5) ;
        \draw (9, 6.5) node[right] {$D_{6}$};

        \draw [->] ( 9, 5) -- ( 11, 7);
        \draw (11,7) node[right] {Type 1};
        \draw [->] ( 9, 5) -- ( 11, 3.5);
        \draw (11,3.5) node[right] {Type 2};
        
        \draw [dashed] ( 12, 10) -- ( 15, 7);
        \draw ( 15, 7) node[right] {$D_0$};
        \draw [dashed] ( 13, 10) -- ( 15.5, 7.5);
        \draw ( 15.5, 7.5) node[right] {$D_{2}$};
        \draw [dashed] ( 14, 10) -- ( 16, 8);
        \draw ( 16, 8) node[right] {$D_{4}$};
        \draw [dashed] ( 15, 10) -- ( 16.5, 8.5);
        \draw ( 16.5, 8.5) node[right] {$D_{6}$};
        \draw [dashed] ( 12, 9) -- ( 14.5, 6.5);
        \draw ( 14.5, 6.5) node[right] {$D_{-2}$};
        \draw [dashed] ( 12, 8) -- ( 14, 6);
        \draw ( 14, 6) node[right] {$D_{-4}$};
        
        \draw [dashed] ( 12, 3) -- ( 15, 0);
        \draw ( 15, 0) node[right] {$D_0$};
        \draw [dashed] ( 13, 3) -- ( 15.5, .5);
        \draw ( 15.5, .5) node[right] {$D_{2}$};
        \draw [dashed] ( 14, 3) -- ( 16, 1);
        \draw ( 16, 1) node[right] {$D_{4}$};
        \draw [dashed] ( 15, 3) -- ( 16.5, 1.5);
        \draw ( 16.5, 1.5) node[right] {$D_{6}$};
        \draw [dashed] ( 12, 2) -- ( 14.5, -.5);
        \draw ( 14.5, -.5) node[right] {$D_{-2}$};
        \draw [dashed] ( 12, 1) -- ( 14, -1);
        \draw ( 14, -1) node[right] {$D_{-4}$};
        
        \draw ( 12.5, 9.5) node {1'};
        \draw ( 13.5, 9.5) node {1};
        \draw ( 12.5, 8.5) node {X};
        \draw ( 13.5, 8.5) node {3};
        
        \draw ( 12.5, 2.5) node {1};
        \draw ( 13.5, 2.5) node {2'};
        \draw ( 14.5, 2.5) node {3'};
        \draw ( 15.5, 2.5) node {3};
        \draw ( 12.5, 1.5) node {X};
        \draw ( 13.5, 1.5) node {2'};
        \draw ( 14.5, 1.5) node {3'};
        \draw ( 15.5, 1.5) node {4};
        \draw ( 12.5, 0.5) node {X};
        \draw ( 13.5, 0.5) node {X};
        \draw ( 14.5, 0.5) node {3'};

        \draw [->] ( 18, 8) -- ( 19, 8);
        \draw [->] ( 18, 1) -- ( 19, 1);
        
        \node[bsq] (17)   at ( 20.5, 8.5) {1'};
        \node[bsq] (18)  [right = of 17] {1};
        \node[bsq] (19)   at ( 20.5, 7.5) {X};
        \node[bsq] (20)  [right = of 19] {3};
        \draw ( 22.5, 8) node {\ $ = t_1$};
        
        \node[bsq] (21)   at ( 20.5, 1.5) {1};
        \node[bsq] (22)  [right = of 21] {2'};
        \node[bsq] (23)  [right = of 22] {3'};
        \node[bsq] (24)  [right = of 23] {3};
        \node[bsq] (25)  [below = of 21] {X};
        \node[bsq] (26)  [right = of 25] {2'};
        \node[bsq] (27)  [right = of 26] {3'};
        \node[bsq] (28)  [right = of 27] {4};
        \node[bsq] (29)  [below = of 25] {X};
        \node[bsq] (30)  [right = of 29] {X};
        \node[bsq] (31)  [right = of 30] {3'};
        \draw ( 24.5, 1) node {\ $ = t_2$};

\end{tikzpicture}
}
\end{center}


We see that $\Gamma_s^{-1}$ is also very similar to $\Gamma^{-1}$. It takes as input a pair of shifted Young tableaux and outputs a shifted domino tableau such that $\Gamma_s^{-1}(\Gamma_s(T))$ is equivalent to $T$. However, we need to describe how the algorithm deals with the cells containing $X$.

Suppose we apply $\Gamma_s^{-1}$ to a pair of shifted Young tableaux $(t_1, t_2)$. At step $i$,  we have the pair $(t_1^{(i-1)}, t_2^{(i-1)})$ of shifted Young tableaux, and $u_i$ the smallest integer (with respect to the relation $1'<1<2'<2<\ldots $) appearing in $(t_1, t_2)$ that doesn't appear in $(t_1^{(i-1)}, t_2^{(i-1)})$. If a cell of $t_1$ (resp. $t_2$) containing $u_i$ has cells filled with $X$ to its left, then all those cells are also added into $t_1^{(i)}$ (resp. $t_2^{(i)}$), while preserving their original positions. This ensures that at every step of the procedure, the constructed tableaux $(t_1^{(i)}, t_2^{(i)})$ are shifted Young tableaux. 

For example, let's apply $\Gamma_s^{-1}$ to the following pair of shifted Young tableaux.

\begin{center}
$(t_1, t_2) = 
\left(
\raisebox{.1in}{
\ \begin{ytableau} 1' & 1 \\ X & 3 \end{ytableau}\  , \  \begin{ytableau} 1 & 2' & 3' & 3 \\ X & 2' & 3' & 4 \\ X & X & 3' \end{ytableau} \ }
\right)
$
\end{center}

\ytableausetup{boxsize=.18in}
\begin{minipage}{.53\textwidth}
\raisebox{.1in}{$(1)$ $\left( \ \begin{ytableau} \scriptstyle{1'} \end{ytableau} , \ \varnothing \right) \rightarrow $ }
\resizebox{.58cm}{!}{
\begin{tikzpicture}[node distance=0 cm,outer sep = 0pt]
        \tikzstyle{every node}=[font=\huge]
        \node[bver] (1) at ( 1, 8) {1'};
\end{tikzpicture}
}\end{minipage} \begin{minipage}{.53\textwidth}
\raisebox{.1in}{$(2)$ $ \left( \ \begin{ytableau} \scriptstyle{1'} & \scriptstyle{1} \end{ytableau} , \ \begin{ytableau} \scriptstyle{1} \end{ytableau} \ \right) \rightarrow $ }
\resizebox{1.3875cm}{!}{
\begin{tikzpicture}[node distance=0 cm,outer sep = 0pt]
        \tikzstyle{every node}=[font=\huge]
        \node[bver] (1) at ( 1, 8) {1'};
        \node[bver] (2) at ( 2, 8) {1};
        \node[bver] (3) at ( 3, 8) {1};

\end{tikzpicture}
}\end{minipage}

\begin{minipage}{.53\textwidth}
\raisebox{.22in}{$(3)$ $ \left( \raisebox{.05in}{ \ \begin{ytableau} \scriptstyle{1'} & \scriptstyle{1} \end{ytableau}, \ \begin{ytableau} \scriptstyle{1} & \scriptstyle{2'} \\ \scriptstyle{X} & \scriptstyle{2'} \end{ytableau} \ } \right) \rightarrow $ }
\resizebox{1.88cm}{!}{
\begin{tikzpicture}[node distance=0 cm,outer sep = 0pt]
        \tikzstyle{every node}=[font=\huge]
        \node[bver] (1) at ( 1, 8) {1'};
        \node[bver] (2) at ( 2, 8) {1};
        \node[bver] (3) at ( 3, 8) {1};
        \node[bver] (4) at ( 4, 8) {2'};
        \node[bhor] (12) at ( 3.5, 6.5) {2'};
        \node[bhor] (10) at ( 1.5, 6.5) {$X$};

\end{tikzpicture}
}\end{minipage} \begin{minipage}{.53\textwidth}
\raisebox{.4in}{$(4)$ $ \left( \raisebox{.1in}{ \ \begin{ytableau} \scriptstyle{1'} & \scriptstyle{1} \end{ytableau}, \ \begin{ytableau} \scriptstyle{1} & \scriptstyle{2'} & \scriptstyle{3'} \\ \scriptstyle{X} & \scriptstyle{2'} & \scriptstyle{3'} \\ \scriptstyle{X} & \scriptstyle{X} & \scriptstyle{3'} \end{ytableau} \ } \right) \rightarrow $ }
\resizebox{2.775cm}{!}{
\begin{tikzpicture}[node distance=0 cm,outer sep = 0pt]
        \tikzstyle{every node}=[font=\huge]
        \node[bver] (1) at ( 1, 8) {1'};
        \node[bver] (2) at ( 2, 8) {1};
        \node[bver] (3) at ( 3, 8) {1};
        \node[bver] (4) at ( 4, 8) {2'};
        \node[bhor] (12) at ( 3.5, 6.5) {2'};
        \node[bhor] (5) at ( 5.5, 8.5) {3'};
        \node[bver] (8) at ( 5, 7) {3'};
        \node[bhor] (13) at ( 4.5, 5.5) {3'};
        \node[bhor] (10) at ( 1.5, 6.5) {$X$};
        \node[bhor] (11) at ( 2.5, 5.5) {$X$};
        \node[bver] (14) at ( 1, 5) {$X$};
        
\end{tikzpicture}
}\end{minipage}

\begin{center}

\begin{minipage}{.55\textwidth}
\raisebox{.4in}{$(5)$ $ \left( 
\raisebox{.1in}{ \ \begin{ytableau} \scriptstyle{1'} & \scriptstyle{1} \\ \scriptstyle{X} & \scriptstyle{3} \end{ytableau}, \ \begin{ytableau} \scriptstyle{1} & \scriptstyle{2'} & \scriptstyle{3'} & \scriptstyle{3} \\ \scriptstyle{X} & \scriptstyle{2'} & \scriptstyle{3'} \\ \scriptstyle{X} & \scriptstyle{X} & \scriptstyle{3'} \end{ytableau} \ }
\right) \rightarrow $ }
\resizebox{3.7cm}{!}{
\begin{tikzpicture}[node distance=0 cm,outer sep = 0pt]
        \tikzstyle{every node}=[font=\huge]
        \node[bver] (1) at ( 1, 8) {1'};
        \node[bver] (2) at ( 2, 8) {1};
        \node[bver] (3) at ( 3, 8) {1};
        \node[bver] (4) at ( 4, 8) {2'};
        \node[bhor] (5) at ( 5.5, 8.5) {3'};
        \node[bhor] (6) at ( 7.5, 8.5) {3};
        \node[bver] (8) at ( 5, 7) {3'};
        \node[bhor] (10) at ( 1.5, 6.5) {$X$};
        \node[bhor] (11) at ( 2.5, 4.5) {X};
        \node[bhor] (12) at ( 3.5, 6.5) {2'};
        \node[bhor] (13) at ( 4.5, 5.5) {3'};
        \node[bhor] (11) at ( 2.5, 5.5) {$X$};
        \node[bver] (14) at ( 1, 5) {$X$};
        \node[bhor] (16) at ( 4.5, 4.5) {3};
\end{tikzpicture}
}\end{minipage} 

\begin{minipage}{.55\textwidth}
\raisebox{.4in}{$(6)$ $ \left( \raisebox{.1in}{ \ \begin{ytableau} \scriptstyle{1'} & \scriptstyle{1} \\ \scriptstyle{X} & \scriptstyle{3} \end{ytableau}, \ \begin{ytableau} \scriptstyle{1} & \scriptstyle{2'} & \scriptstyle{3'} & \scriptstyle{3} \\ \scriptstyle{X} & \scriptstyle{2'} & \scriptstyle{3'} & \scriptstyle{4} \\ \scriptstyle{X} & \scriptstyle{X} & \scriptstyle{3'} \end{ytableau} \ } \right) \rightarrow $ }
\resizebox{3.7cm}{!}{
\begin{tikzpicture}[node distance=0 cm,outer sep = 0pt]
        \tikzstyle{every node}=[font=\huge]
        \node[bver] (1) at ( 1, 8) {1'};
        \node[bver] (2) at ( 2, 8) {1};
        \node[bver] (3) at ( 3, 8) {1};
        \node[bver] (4) at ( 4, 8) {2'};
        \node[bhor] (5) at ( 5.5, 8.5) {3'};
        \node[bhor] (6) at ( 7.5, 8.5) {3};
        \node[bhor] (7) at ( 6.5, 7.5) {4};
        \node[bver] (8) at ( 5, 7) {3'};
        \node[bhor] (10) at ( 1.5, 6.5) {$X$};
        \node[bhor] (11) at ( 2.5, 4.5) {X};
        \node[bhor] (12) at ( 3.5, 6.5) {2'};
        \node[bhor] (13) at ( 4.5, 5.5) {3'};
        \node[bhor] (11) at ( 2.5, 5.5) {$X$};
        \node[bver] (14) at ( 1, 5) {$X$};
        \node[bhor] (16) at ( 4.5, 4.5) {3};
\end{tikzpicture}
} \end{minipage}

\end{center}
Notice that we do not recover exactly the tableau $T$ that we started with. Instead, we obtain a shifted domino tableau equivalent to $T$.

As a corollary, we then have the following. 

\begin{corollary}[Theorem 3.2 \cite{chemli}]
Let $\lambda$ be a shifted pavable partition with 2-quotient $(\mu,\nu)$. One has 
\[Q_\mu Q_\nu=\sum_{sh(T)=\lambda}x^T,\]
where we sum over the set of shifted domino tableaux of shape $\lambda$. 
\end{corollary}

\begin{proof}
Each term in the product $Q_\mu Q_\nu$ is represented by a pair of shifted tableaux of shape $(\mu,\nu)$. Theorem~\ref{thm:bijectiondomino4} says that the set of these pairs are in bijection with the set of shifted domino tableaux of shape $\lambda$, and the shifted domino tableau associated to a pair of shifted tableaux $(t_1,t_2)$ has the same multiset of entries as $(t_1,t_2)$.
\end{proof}

\section{Shifted $K$-theoretic generalizations}\label{sec:shiftedK}
\subsection{$K$-theoretic $Q$-Schur functions}
There is a natural $K$-theoretic analogue of the $Q$-Schur functions, introduced by Ikeda--Naruse \cite{ikeda2013k} and Graham--Kreiman \cite{graham2015excited}, called the \textit{$K$-theoretic $Q$-Schur function} and denoted $GQ_\lambda$. In fact, Ikeda and Naruse introduced a more general \textit{$K$-theoretic factorial $Q$-Schur function}, but it will suffice for us to consider the restricted generality. As a natural $K$-theoretic analogue, the $GQ_\lambda$ are related to the $K$-theory of the maximal isotropic Grassmannian of symplectic type.

Using the same linear order on subsets of $\{1'<1<2'<2<\ldots\}$, where $A\leq B$ if $\max(A)\leq\max(B)$, we have the following definition.

\begin{definition}\cite{ikeda2013k}
Let $\lambda=(\lambda_1,\ldots,\lambda_k)$ be a partition with $\lambda_k\geq k$. A \textit{shifted set-valued  Young tableau} of shape $\lambda$ is a filling of the cells of $down(\lambda)$ with $X$ and the cells of $up(\lambda)$ with finite, non-empty sets of positive integers such that 
\begin{itemize}
\item entries weakly increase across rows and down columns,
\item there is at most one occurrence of $i$ in any column and 
\item there is at most one occurrence of $i'$ in any row. 
\end{itemize}
Let $up(T)$ denote the cells of $T$ weakly above $D_0$ along with their filling.
\end{definition}

We associate to each shifted set-valued tableau $T$ a monomial $x^T$,
\[x^T=x_1^{\beta_1(T)}x_2^{\beta_2(T)}\cdots,\] where again $\beta_i(T)$ is the number of occurrences of $i$ and $i'$ in $T$. We also let $|T| = |up(T)|$ denote the number of primed and unprimed integers in $T$. We define the diagonal reading word of shifted set-valued tableau $T$ in the natural way. It is easy to see that the diagonal reading word uniquely defines the shifted set-valued Young tableau.

\begin{definition}\cite{ikeda2013k}
Let $\lambda=(\lambda_1,\ldots,\lambda_k)$ be a partition with $\lambda_k\geq k$. 
The \textit{$K$-theoretic $Q$-Schur function} $GQ_\lambda$ is
\[GQ_\lambda=\sum_{sh(T)=\lambda}(-1)^{|up(T)|-|up(\lambda)|}x^T, \]
where we sum over all shifted set-valued tableaux of shape $\lambda$ and $|up(\lambda)|$ denotes the number of boxes in $up(\lambda)$.
\end{definition}

\begin{example}
We may compute some monomials of $GQ_{(2,2)}$ using the tableaux shown below. 
\[GQ_{(2,2)}=4x_1^2x_2-2x_1^3x_2-4x_1x_2^2x_3-x_1x_2^2x_3^2x_5\pm \cdots\]
\begin{center}
\ytableausetup{boxsize=.24in}
\begin{ytableau}
\scriptstyle{1'} & \scriptstyle{1} \\
\scriptstyle{X} & \scriptstyle{2}
\end{ytableau}
\begin{ytableau}
\scriptstyle{1} & \scriptstyle{1} \\
\scriptstyle{X} & \scriptstyle{2}
\end{ytableau}
\begin{ytableau}
\scriptstyle{1'} & \scriptstyle{1} \\
\scriptstyle{X} & \scriptstyle{2'}
\end{ytableau}
\begin{ytableau}
\scriptstyle{1} & \scriptstyle{1} \\
\scriptstyle{X} & \scriptstyle{2'}
\end{ytableau}
\begin{ytableau}
\scriptstyle{1',1} & \scriptstyle{1} \\
\scriptstyle{X} & \scriptstyle{2}
\end{ytableau}
\begin{ytableau}
\scriptstyle{1',1} & \scriptstyle{1} \\
\scriptstyle{X} & \scriptstyle{2'}
\end{ytableau}
\begin{ytableau}
\scriptstyle{1,2} & \scriptstyle{2} \\
\scriptstyle{X} & \scriptstyle{3}
\end{ytableau}
\begin{ytableau}
\scriptstyle{1',2'} & \scriptstyle{2} \\
\scriptstyle{X} & \scriptstyle{3}
\end{ytableau}
\begin{ytableau}
\scriptstyle{1'} & \scriptstyle{2'} \\
\scriptstyle{X} & \scriptstyle{2,3}
\end{ytableau}
\begin{ytableau}
\scriptstyle{1'} & \scriptstyle{2'} \\
\scriptstyle{X} & \scriptstyle{2',3}
\end{ytableau}
\begin{ytableau}
\scriptstyle{1',2'} & \scriptstyle{2,3'} \\
\scriptstyle{X} & \scriptstyle{3',5}
\end{ytableau}
\end{center}
Note that we have not listed all tableaux with monomials $x_1x_2^2x_3$ and $x_1x_2^2x_3^2x_5$. We see that the diagonal reading word for the rightmost tableau above is $$\{1',2'\},\{3',5\}\ /\ \{2,3'\}.$$ 
\end{example}
Since shifted Young tableaux are shifted set-valued tableaux, we see that the lowest degree terms of $GQ_\lambda$ make up $Q_\lambda$.

\subsection{Shifted set-valued domino tableaux}


\begin{definition}
Let $\lambda$ be a shifted pavable partition. A \textit{shifted set-valued domino tableau} is a filling of the dominoes of $down(\lambda)$ with $X$ and the dominoes of $up(\lambda)$ with finite, nonempty sets of primed and unprimed integers with linear order $\{1'<1<2'<2<\cdots\}$ such that: 
\begin{enumerate}
\item Restricting to the minimum entry in each domino yields a shifted domino tableau.
\item If $F_1$ and $F_2$ are dominoes of the same type on neighboring diagonals and $F_2$ is weakly southeast of $F_1$, then $\max(F_1)\leq\min(F_2)$, and
\begin{itemize}
\item $\max(F_1)<\min(F_2)$ if $F_1$ is located on $D_{2k}$, $F_2$ is on $D_{2(k+1)}$, and $\max(F_1)$ is primed,  and
\item $\max(F_1)<\min(F_2)$ if $F_1$ is located on $D_{2(k+1)}$, $F_2$ is on $D_{2k}$, and $\max(F_1)$ is unprimed. 
\end{itemize}
\end{enumerate}


For $T$ a shifted set-valued domino tableau, let $up(T)$ be the dominoes of $T$ that lie on a diagonal weakly northeast of $D_0$ along with the filling of these dominoes. We consider two shifted set-valued tableaux $T$ and $T'$ to be \textit{equivalent} if $up(T)=up(T')$. The \textit{set of shifted set-valued domino tableaux} refers to the set up to equivalence.
\end{definition}

We define the diagonal reading word in the natural way and again note that equivalent shifted set-valued domino tableaux have equal diagonal reading words.

\begin{example}
Below are two equivalent shifted set-valued domino tableaux of shape $(6, 5, 5, 5, 3)$ with diagonal reading word $\{ 1, 2 \}, 1, 3',\{ 4', 7\}, 4' \ / \ \{ 1, 2'\}, 3', \{3, 4'\} \ / \ 2$. 

\begin{center}
\resizebox{3.8cm}{!}{
\begin{tikzpicture}[node distance=0 cm,outer sep = 0pt]
        \tikzstyle{every node}=[font=\Large]
        \node[bver] (1) at ( 1, 8) {1, 2};
        \node[bver] (2) at ( 2, 8) {1};
        \node[bhor] (5) at ( 3.5, 8.5) {1,2'};
        \node[bhor] (6) at ( 5.5, 8.5) {2};
        \node[bhor] (7) at ( 3.5, 7.5) {3'};
        \node[bhor] (12) at ( 1.5, 6.5) {X};
        \node[bhor] (13) at ( 3.5, 6.5) {3'};
        \node[bver] (14) at ( 1, 5) {X};
        \node[bhor] (15) at (2.5, 5.5) {X};
        \node[bver] (16) at (5, 7) {3,4'};
        \node[bhor] (17) at (4.5, 5.5) {4',7};
        \node[bhor] (18) at (2.5, 4.5) {X};
        \node[bhor] (19) at (4.5, 4.5) {4'};
\end{tikzpicture}
}
\hspace{1in}
\resizebox{3.8cm}{!}{
\begin{tikzpicture}[node distance=0 cm,outer sep = 0pt]
        \tikzstyle{every node}=[font=\Large]
        \node[bver] (1) at ( 1, 8) {1,2};
        \node[bver] (2) at ( 2, 8) {1};
        \node[bhor] (5) at ( 3.5, 8.5) {1,2'};
        \node[bhor] (6) at ( 5.5, 8.5) {2};
        \node[bhor] (7) at ( 3.5, 7.5) {3'};
        \node[bver] (12) at ( 1, 6) {X};
        \node[bhor] (13) at ( 3.5, 6.5) {3'};
        \node[bhor] (14) at ( 1.5, 4.5) {X};
        \node[bver] (15) at (2, 6) {X};
        \node[bver] (16) at (5, 7) {3,4'};
        \node[bhor] (17) at (4.5, 5.5) {4',7};
        \node[bver] (18) at (3, 5) {X};
        \node[bhor] (19) at (4.5, 4.5) {4'};
\end{tikzpicture}
}
\end{center} 

\end{example}

We may now state the main result of this section. 

\begin{theorem} Let $\lambda$ be a shifted pavable partition with 2-quotient $(\mu,\nu)$. The set of shifted set-valued domino tableaux of shape $\lambda$ is in bijection with the set of pairs $(t_1,t_2)$ of shifted set-valued tableaux of shape $(\mu,\nu)$.
\label{thm:bijectiondomino3}
\end{theorem}


\begin{proof}
We define a set-valued version of $\Gamma_s$ and $\Gamma_s^{-1}$ called $\Gamma_s^*$ and $\Gamma_s^{*-1}$ analogously to the definition of $\Gamma^*$ and $\Gamma^{*-1}$ from $\Gamma$ and $\Gamma^{-1}$. See Example~\ref{ex:shiftedsetexample} for an illustration.

Let $T$ be a shifted set-valued domino tableau. We show that $\Gamma_s^*(T)=(t_1,t_2)$ is a pair of shifted set-valued Young tableaux. We can use the same argument as in the proof of Theorem~\ref{thm:bijectiondomino2} to show that $t_1$ and $t_2$ are weakly increasing in rows and columns. To see there is at most one occurrence of $i'$ in a row, suppose $b_1$ lies directly left of $b_2$ in $t_i$ and let $F_1$  and $F_2$ be the dominoes of $T$ such that $\Gamma_s^*$ sends the entries of $F_1$ to $b_1$ and those of $F_2$ to $b_2$. Using the argument from the proof of Theorem~\ref{thm:bijectiondomino2}, we know $F_2$ is weakly southeast of $F_1$. Hence if $\max(b_1)$ is primed, $\max(F_1)<\min(F_2)$, and so $\max(b_1)<\min(b_2)$. We can similarly argue that there is at most one occurrence of $i$ in any column of $t_i$.



Now let $(t_1, t_2)$ be a pair of shifted set-valued Young tableaux. We show that $T = \Gamma_s^{*-1}(t_1, t_2)$ is a shifted set-valued domino tableau. Using an argument analogous to that in the proof of Theorem~\ref{thm:bijectiondomino2}, we see that restricting to the minimum entry in each domino yields a shifted domino tableau and that $\max(F_1)\leq \min(F_2)$ when $F_2$ is weakly southeast of $F_1$.

Suppose that $F_2$ is weakly southeast of $F_1$, $F_1$ is on a diagonal $D_{2k}$ and $F_2$ is on $D_{2(k+1)}$, and that $\max(F_1)$ is primed. Let $b_1$ and $b_2$ be the boxes of $t_i$ such that $\Gamma_s^{*-1}$ sends the entries of $b_1$ to $F_1$ and the entries of $b_2$ to $F_2$. We have shown in the proof of Theorem~\ref{thm:bijectiondomino2} that $b_2$ must be weakly southeast of $b_1$. Since $\max(F_1)$ is primed, $\max(b_1)$ is primed. Then $\max(b_1)<\min(b_2)$ because $t_i$ is a shifted set-valued tableau, and so $\max(F_1)<\min(F_2)$. We can similarly show that $\max(F_1)<\min(F_2)$ if $F_1$ is on $D_{2(k+1)}$, $F_2$ is on $D_{2k}$, and $\max(F_1)$ is unprimed. 

It is clear that $\Gamma_s^*$ and $\Gamma_s^{*-1}$ are inverses because $\Gamma_s$ and $\Gamma_s^{-1}$ are.


\end{proof}

\begin{example}\label{ex:shiftedsetexample} Let $T$ be the following shifted set-valued domino tableau. We apply $\Gamma_s^*$ to the tableau $T$.

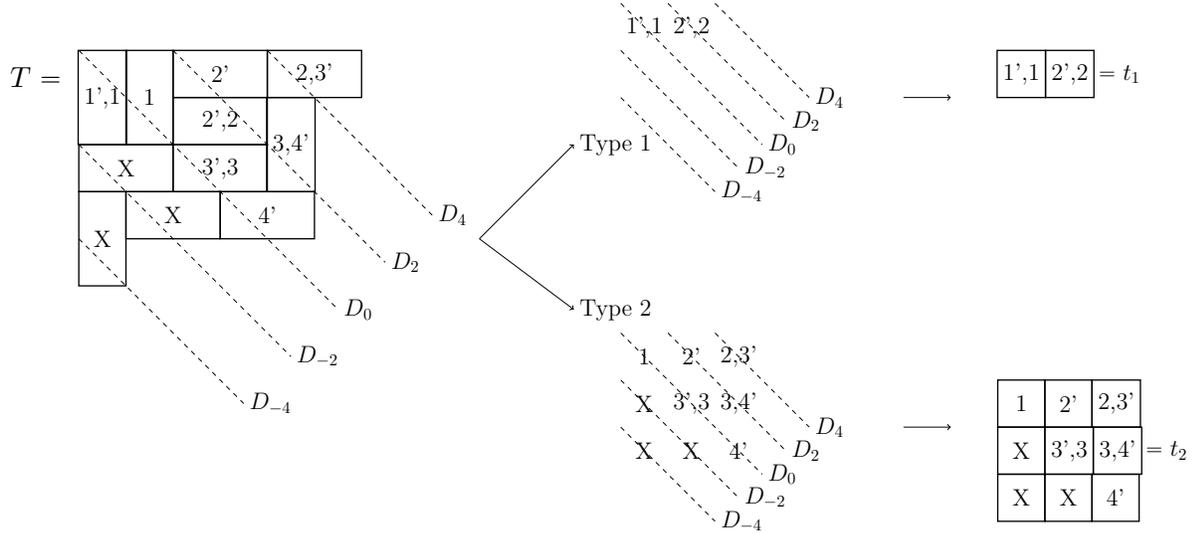
\begin{center}
\raisebox{6cm}{$T = $}
\resizebox{15cm}{!}{
\begin{tikzpicture}[node distance=0 cm,outer sep = 0pt]

        \tikzstyle{every node}=[font=\Large]
        \node[bver] (1) at ( 1, 8) {1',1};
        \node[bver] (2) at ( 2, 8) {1};
        \node[bhor] (5) at ( 3.5, 8.5) {2'};
        \node[bhor] (6) at ( 5.5, 8.5) {2,3'};
        \node[bhor] (7) at ( 3.5, 7.5) {2',2};
        \node[bhor] (12) at ( 1.5, 6.5) {X};
        \node[bhor] (13) at ( 3.5, 6.5) {3',3};
        \node[bver] (14) at ( 1, 5) {X};
        \node[bhor] (15) at (2.5, 5.5) {X};
        \node[bver] (16) at (5, 7) {3,4'};
        \node[bhor] (17) at (4.5, 5.5) {4'};
        
        \draw[dashed] ( 0.5, 9) -- ( 6, 3.5) ;
        \draw (6, 3.5) node[right] {$D_0$};
        \draw[dashed] ( 0.5, 7) -- ( 5, 2.5) ;
        \draw (5, 2.5) node[right] {$D_{-2}$};
        \draw[dashed] ( 0.5, 5) -- ( 4, 1.5) ;
        \draw (4, 1.5) node[right] {$D_{-4}$};
        \draw[dashed] ( 2.5, 9) -- ( 7, 4.5) ;
        \draw (7, 4.5) node[right] {$D_{2}$};
        \draw[dashed] ( 4.5, 9) -- ( 8, 5.5) ;
        \draw (8, 5.5) node[right] {$D_{4}$};

        \draw [->] ( 9, 5) -- ( 11, 7);
        \draw (11,7) node[right] {Type 1};
        \draw [->] ( 9, 5) -- ( 11, 3.5);
        \draw (11,3.5) node[right] {Type 2};
        
        \draw [dashed] ( 12, 10) -- ( 15, 7);
        \draw ( 15, 7) node[right] {$D_0$};
        \draw [dashed] ( 13, 10) -- ( 15.5, 7.5);
        \draw ( 15.5, 7.5) node[right] {$D_{2}$};
        \draw [dashed] ( 14, 10) -- ( 16, 8);
        \draw ( 16, 8) node[right] {$D_{4}$};
        \draw [dashed] ( 12, 9) -- ( 14.5, 6.5);
        \draw ( 14.5, 6.5) node[right] {$D_{-2}$};
        \draw [dashed] ( 12, 8) -- ( 14, 6);
        \draw ( 14, 6) node[right] {$D_{-4}$};
        
        \draw [dashed] ( 12, 3) -- ( 15, 0);
        \draw ( 15, 0) node[right] {$D_0$};
        \draw [dashed] ( 13, 3) -- ( 15.5, .5);
        \draw ( 15.5, .5) node[right] {$D_{2}$};
        \draw [dashed] ( 14, 3) -- ( 16, 1);
        \draw ( 16, 1) node[right] {$D_{4}$};
        \draw [dashed] ( 12, 2) -- ( 14.5, -.5);
        \draw ( 14.5, -.5) node[right] {$D_{-2}$};
        \draw [dashed] ( 12, 1) -- ( 14, -1);
        \draw ( 14, -1) node[right] {$D_{-4}$};
        
        \draw ( 12.5, 9.5) node {1',1};
        \draw ( 13.5, 9.5) node {2',2};
        
        \draw ( 12.5, 2.5) node {1};
        \draw ( 13.5, 2.5) node {2'};
        \draw ( 14.5, 2.5) node {2,3'};
        \draw ( 12.5, 1.5) node {X};
        \draw ( 13.5, 1.5) node {3',3};
        \draw ( 14.5, 1.5) node {3,4'};
        \draw ( 12.5, 0.5) node {X};
        \draw ( 13.5, 0.5) node {X};
        \draw ( 14.5, 0.5) node {4'};

        \draw [->] ( 18, 8) -- ( 19, 8);
        \draw [->] ( 18, 1) -- ( 19, 1);
        
        \node[bsq] (17)   at ( 20.5, 8.5) {1',1};
        \node[bsq] (18)  [right = of 17] {2',2};
        \draw ( 22.5, 8.5) node {\ $ = t_1$};
        
        \node[bsq] (21)   at ( 20.5, 1.5) {1};
        \node[bsq] (22)  [right = of 21] {2'};
        \node[bsq] (23)  [right = of 22] {2,3'};
        \node[bsq] (25)  [below = of 21] {X};
        \node[bsq] (26)  [right = of 25] {3',3};
        \node[bsq] (27)  [right = of 26] {3,4'};
        \node[bsq] (29)  [below = of 25] {X};
        \node[bsq] (30)  [right = of 29] {X};
        \node[bsq] (31)  [right = of 30] {4'};
        \draw ( 23.5, 0.5) node {\ $ = t_2$};

\end{tikzpicture}
}
\end{center}

If we take the pair of shifted set-valued tableaux $(t_1, t_2)$, we will see that we reconstruct $T$ exactly the same way as in Section~\ref{sec:shiftedbijection}, with integer set entries instead of integers.

\ytableausetup{boxsize=.23in}

\begin{center}
$(t_1, t_2) = 
\left( \raisebox{.15in}{ \begin{ytableau} \scriptstyle{1',1} & \scriptstyle{2',2} \end{ytableau}\  , \  \begin{ytableau} \scriptstyle{1} & \scriptstyle{2'} & \scriptstyle{2,3'} \\ \scriptstyle{X} & \scriptstyle{3',3} & \scriptstyle{3,4'} \\ \scriptstyle{X} & \scriptstyle{X} & \scriptstyle{4'} \end{ytableau} \ } \right) $
\end{center}

\begin{minipage}{.6\textwidth}
\raisebox{.2in}{$(1)$ $\left( \ \begin{ytableau} \scriptstyle{1',1} \end{ytableau} , \ \varnothing \right) \rightarrow $ }
\resizebox{0.85cm}{!}{
\begin{tikzpicture}[node distance=0 cm,outer sep = 0pt]
        \tikzstyle{every node}=[font=\Large]
        \node[bver] (1) at ( 1, 8) {1',1};
\end{tikzpicture}
}\end{minipage} 
\begin{minipage}{.52\textwidth}
\raisebox{.2in}{$(2)$ $ \left( \ \begin{ytableau} \scriptstyle{1',1} \end{ytableau} , \ \begin{ytableau} \scriptstyle{1} \end{ytableau} \ \right) \rightarrow $ }
\resizebox{1.4cm}{!}{
\begin{tikzpicture}[node distance=0 cm,outer sep = 0pt]
        \tikzstyle{every node}=[font=\Large]
        \node[bver] (1) at ( 1, 8) {1',1};
        \node[bver] (2) at ( 2, 8) {1};
\end{tikzpicture}
}\end{minipage}

\begin{center}

\begin{minipage}{.52\textwidth}
\raisebox{.2in}{$(3)$ $ \left( \ \begin{ytableau} \scriptstyle{1',1} & \scriptstyle{2',2} \end{ytableau}, \ \begin{ytableau} \scriptstyle{1} & \scriptstyle{2'} \end{ytableau} \ \right) \rightarrow $ }
\resizebox{2.6cm}{!}{
\begin{tikzpicture}[node distance=0 cm,outer sep = 0pt]
        \tikzstyle{every node}=[font=\Large]
        \node[bver] (1) at ( 1, 8) {1',1};
        \node[bver] (2) at ( 2, 8) {1};
        \node[bhor] (5) at ( 3.5, 8.5) {2'};
        \node[bhor] (7) at ( 3.5, 7.5) {2',2};
\end{tikzpicture}
}\end{minipage} 

\begin{minipage}{.6\textwidth}
\raisebox{.2in}{$(4)$ $ \left( \ \begin{ytableau} \scriptstyle{1',1} & \scriptstyle{2',2}  \end{ytableau}, \ \begin{ytableau} \scriptstyle{1} & \scriptstyle{2'} & \scriptstyle{2,3'} \end{ytableau} \ \right) \rightarrow $ }
\resizebox{3.7cm}{!}{
\begin{tikzpicture}[node distance=0 cm,outer sep = 0pt]
        \tikzstyle{every node}=[font=\Large]
        \node[bver] (1) at ( 1, 8) {1',1};
        \node[bver] (2) at ( 2, 8) {1};
        \node[bhor] (5) at ( 3.5, 8.5) {2'};
        \node[bhor] (6) at ( 5.5, 8.5) {2,3'};
        \node[bhor] (7) at ( 3.5, 7.5) {2',2};

\end{tikzpicture}
}\end{minipage}

\begin{minipage}{.6\textwidth}
\raisebox{0.35in}{$(5)$ $ \left( \ \begin{ytableau} \scriptstyle{1',1} & \scriptstyle{2',2}  \end{ytableau}, \ \begin{ytableau} \scriptstyle{1} & \scriptstyle{2'} & \scriptstyle{2,3'} \\ \scriptstyle{X} & \scriptstyle{3',3} \end{ytableau} \ \right) \rightarrow $ }
\resizebox{3.7cm}{!}{
\begin{tikzpicture}[node distance=0 cm,outer sep = 0pt]
        \tikzstyle{every node}=[font=\Large]
        \node[bver] (1) at ( 1, 8) {1',1};
        \node[bver] (2) at ( 2, 8) {1};
        \node[bhor] (5) at ( 3.5, 8.5) {2'};
        \node[bhor] (6) at ( 5.5, 8.5) {2,3'};
        \node[bhor] (7) at ( 3.5, 7.5) {2',2};
        \node[bhor] (12) at ( 1.5, 6.5) {X};
        \node[bhor] (13) at ( 3.5, 6.5) {3',3};
        
\end{tikzpicture}
}\end{minipage}

\begin{minipage}{.6\textwidth}
\raisebox{0.35in}{$(6)$ $ \left( \ \begin{ytableau} \scriptstyle{1',1} & \scriptstyle{2',2}  \end{ytableau}, \ \begin{ytableau} \scriptstyle{1} & \scriptstyle{2'} & \scriptstyle{2,3'} \\ \scriptstyle{X} & \scriptstyle{3',3} & \scriptstyle{3,4'} \end{ytableau} \ \right) \rightarrow $ }
\resizebox{3.7cm}{!}{
\begin{tikzpicture}[node distance=0 cm,outer sep = 0pt]
        \tikzstyle{every node}=[font=\Large]
        \node[bver] (1) at ( 1, 8) {1',1};
        \node[bver] (2) at ( 2, 8) {1};
        \node[bhor] (5) at ( 3.5, 8.5) {2'};
        \node[bhor] (6) at ( 5.5, 8.5) {2,3'};
        \node[bhor] (7) at ( 3.5, 7.5) {2',2};
        \node[bhor] (12) at ( 1.5, 6.5) {X};
        \node[bhor] (13) at ( 3.5, 6.5) {3',3};
        \node[bver] (16) at (5, 7) {3,4'};
\end{tikzpicture}
}\end{minipage} 

\begin{minipage}{.6\textwidth}
\raisebox{0.65in}{$(7)$ $ \left( \raisebox{.1in}{ \ \begin{ytableau} \scriptstyle{1',1} & \scriptstyle{2',2}  \end{ytableau}, \ \begin{ytableau} \scriptstyle{1} & \scriptstyle{2'} & \scriptstyle{2,3'} \\ \scriptstyle{X} & \scriptstyle{3',3} & \scriptstyle{3,4'} \\ \scriptstyle{X} & \scriptstyle{X} & \scriptstyle{4'} \end{ytableau} \ } \right) \rightarrow $ }
\resizebox{3.7cm}{!}{
\begin{tikzpicture}[node distance=0 cm,outer sep = 0pt]
        \tikzstyle{every node}=[font=\Large]
        \node[bver] (1) at ( 1, 8) {1',1};
        \node[bver] (2) at ( 2, 8) {1};
        \node[bhor] (5) at ( 3.5, 8.5) {2'};
        \node[bhor] (6) at ( 5.5, 8.5) {2,3'};
        \node[bhor] (7) at ( 3.5, 7.5) {2',2};
        \node[bhor] (12) at ( 1.5, 6.5) {X};
        \node[bhor] (13) at ( 3.5, 6.5) {3',3};
        \node[bver] (14) at ( 1, 5) {X};
        \node[bhor] (15) at (2.5, 5.5) {X};
        \node[bver] (16) at (5, 7) {3,4'};
        \node[bhor] (17) at (4.5, 5.5) {4'};
\end{tikzpicture}
}\end{minipage}

\end{center}

\end{example}

\begin{corollary}
Let $\lambda$ be a shifted pavable partition with 2-quotient $(\mu,\nu)$. Then
\[GQ_\mu GQ_\nu = \sum_{sh(T)=\lambda}(-1)^{|up(T)|-|up(\lambda)|}x^T,\]
where we sum over all shifted set-valued domino tableaux of shape $\lambda$, $|up(T)|$ denotes the number of positive integers in $up(T)$ and $|up(\lambda)|$ denotes the number of dominoes in $up(\lambda)$.
\end{corollary}

\begin{proof}
Consider a term in the product $GQ_\mu GQ_\nu$. This monomial corresponds to a pair of shifted set-valued tableaux: $t_1$ of shape $\mu$ and $t_2$ of shape $\nu$. The pair $(t_1,t_2)$ corresponds to some shifted set-valued domino tableau $T$ of shape $\lambda$ by Theorem~\ref{thm:bijectiondomino3}. It is then clear from the previous bijection that $x^{t_1} x^{t_2} = x^T$.

We now examine the sign of $x^{t_1}x^{t_2}$ in $GQ_\mu GQ_\nu$. We see that it appears with sign 
\[(-1)^{|up(t_1)|-|up(\mu)|}(-1)^{|up(t_2)|-|up(\nu)|}=(-1)^{|up(t_1)|+|up(t_2)|-(|up(\mu)|+|up(\nu)|)}=(-1)^{|up(T)|-|up(\lambda)|}.\] This gives the desired result.
\end{proof}

\section*{Acknowledgements}
This paper is the result of the an undergraduate research project, where FM-G was funded by 
an CRM-ISM Summer Undergraduate Scholarship. The authors are grateful to Hugh Thomas for his support and to the LaCIM community. RP received support from NSERC, CRM-ISM, and the Canada Research Chairs Program.
We thank the anonymous reviewer for comments to help improve exposition.
\bibliographystyle{alpha}
\bibliography{main.bib}

\end{document}